\documentclass[11pt]{article}
\usepackage{geometry} 
\usepackage{amsmath,amsthm,amssymb}                
\usepackage{graphicx, color}
\usepackage{amssymb}
\usepackage{epstopdf}
\usepackage{pdfsync}
\usepackage{mathabx}

\newtheorem{definition}{Definition}[section]
\newtheorem{lemma}[definition]{Lemma}
\newtheorem{theorem}[definition]{Theorem}
\newtheorem{proposition}[definition]{Proposition}
\newtheorem{corollary}[definition]{Corollary}
\newtheorem{remark}[definition]{Remark}

\RequirePackage{authblk}

\numberwithin{equation}{section}

\def\e{\varepsilon}

\def\wto{\rightharpoonup}

\title{Homogenization of some viscous Hamilton--Jacobi equations using variational arguments on probability spaces}
\author[1]{Andrea Braides} 
\author[2]{Gianni Dal Maso}
\author[3]{Claude Le Bris}%\footnote{corresponding author, e-mail: {\tt claude.le-bris@enpc.fr}} 
\affil[1]{\small Department of Mathematics, University of Rome Tor Vergata, via della Ricerca Scientifica 1, Rome, Italy {\tt braides@mat.uniroma2.it}}
\affil[2]{\small SISSA, via Bonomea 265, Trieste, Italy {\tt dalmaso@sissa.it}}
\affil[3]{\small \'Ecole Nationale des Ponts et Chauss\'ees, Institut Polytechnique de Paris, CNRS and INRIA MATHERIALS project-team,
77455~Marne-La-Vall\'ee Cedex 2, France,
{\tt claude.le-bris@enpc.fr}}
\date{}
\begin{document} 

\maketitle

\begin{abstract}This work studies the asymptotic behavior of a class of viscous Hamilton--Jacobi equations with oscillatory periodic potentials, already studied by other authors in the PDE framework, and described in terms of an effective Hamiltonian. A novel variational approach is introduced, based on the classical stochastic control representation involving a Brownian motion, which allows the problem to be reformulated in terms of functionals defined on probability spaces. The core result is the identification of a homogenized Lagrangian obtained via a $\Gamma$-convergence procedure adapted to random processes, where filtrations and Brownian motions are treated as variables. The effective Hamiltonian is then recovered as the Legendre transform of this homogenized Lagrangian. The analysis provides both lower and upper bounds by means of blow-up arguments and the construction of stochastic correctors, yielding a complete characterization of the homogenized limit. This approach allows for a stability result with respect to the addition of some localized defects embedded in the periodic background.

{\bf MSC codes:} 49J45, 35F21, 35B40, 35B20, 35B35, 35B27

{\bf Keywords:} Homogenization,  $\Gamma$-convergence, Hamilton--Jacobi equations

\end{abstract}

%\blfootnote{Preprint SISSA  19/2024/MATE}

\def\EE{\mathbb E}

\section{Introduction}
In this work we consider the asymptotic behaviour of the solutions to the oscillatory Hamil\-ton-Jacobi equation \begin{equation}\label{11}
\lambda U_{\e}(x)-\frac\e2 \triangle U_{\e}(x)+H_{\rm per}\Big(\frac x\e,
\nabla U_{\e}(x)\Big)=
\psi(x),
\qquad x\in\mathbb R^d.
\end{equation}
with $\lambda>0$, and the quadratic Hamiltonian
$$
H_{\rm per}(y,p)= \frac12|p|^2-V_{\rm per}(y),
$$ 
where $V_{\rm per}$ is periodic. The functions $V_{\rm per}$ and $\psi$ are supposed sufficiently regular, say Lipschitz continuous, so that, for each $\e>0$, the solution exists and is unique in the sense of viscosity solutions.

As is well known, this asymptotic behaviour can be characterized by an effective {\em homogenized Hamiltonian} $H_{\rm hom}$, for which the solutions   
$U_{\e}$ converge to the viscosity solution of 
\begin{equation}\label{12}
\lambda U(x)+H_{\rm hom}(\nabla U(x))=\psi(x),
\qquad x\in\mathbb R^d.
\end{equation}
as $\e\to 0^+$. In the PDE approach to homogenization, the value of the Hamiltonian $H_{\rm hom}$ at $p$ is characterized by the correctors equation
\begin{equation}\label{uhs}
-\frac12 \triangle w_{p}(x)+H_{\rm per}(x, p+\nabla w_{p}(x))=H_{\rm hom}(p),
\qquad x\in\mathbb R^d;
\end{equation}
that is, the value $H_{\rm hom}(p)$ is the unique constant such that there is a periodic function $w_p$, unique up to addition of a constant, such that  \eqref{uhs} holds in the sense of viscosity solutions (see \cite{MR1159184,PLL-college}, and related work by many authors \cite{Lions-Sou}).

On the other hand, given a probability space $(\Omega, \mathcal F, P)$ and a $d$-dimensional Brownian motion $B_t$ on it,
the solution $U_{\e}$ can also be characterized as the value function
$$
U_{\e}(x)=\inf\Big\{
\int_0^{+\infty} \EE\Big(\frac12 |\dot v(t)|^2+ V_{\rm per} \Big(\frac{\sqrt\e B_t+ v(t)}\e\Big)+\psi(\sqrt\e B_t+ v(t))\Big)\, e^{-\lambda t}dt:  v(0)=x\Big\},
$$
where $\EE$ denotes the expectation with respect to $(\Omega, \mathcal F, P)$.
Here, the infimum is taken over all random functions $v\colon(0,+\infty) \to L^2(\Omega;\mathbb R^d)$ in the corresponding weighted Sobolev space, satisfying the initial condition
and adapted to the filtration corresponding to the Brownian motion (see e.g.~\cite[Section I.8]{MR1484411} and \cite{zbMATH06190335,zbMATH07814183,MR2179357}). It is interesting to remark that the minimum value is independent of the choice of the Brownian motion, since the solution $U_\e$ to equation \eqref{11} is unique.

With this characterization of solutions in mind, we propose a new approach, through a variational homogenization result for the functionals
\begin{equation}\label{fnb}
\int_0^{+\infty} \EE\Big(\frac12 |\dot v(t)|^2+ V_{\rm per} \Big(\frac{\sqrt\e B_t+ v(t)}\e\Big)\Big)\, e^{-\lambda t}dt.
\end{equation}
It is worth noting that a homogenization result for such functionals alone, which is expressed in terms of $\Gamma$-convergence, is not sufficient to study the behaviour of the minimum problems defining $U_\e(x)$, since the additional term 
$$
\int_0^{+\infty} \EE(\psi(\sqrt\e B_t+ v(t)))\, e^{-\lambda t}dt
$$
cannot easily be treated as a continuous perturbation of the main functional \eqref{fnb}.
We adopt a different approach. Using the results established here for  \eqref{fnb}, we characterize $H_{\rm hom}$ as the Legendre transform of a homogenized Lagrangian, for which it is sufficient to consider the case $\psi=0$.  From this we characterize the asymptotic behaviour of $U_\e(x)$ for $\psi\ne 0$ relying on the PDE results in  \cite{MR1159184}.

 \smallskip
We now describe our new type of $\Gamma$-convergence, in which the main conceptual novelty is that filtrations and Brownian motions are treated as optimization variables rather than fixed probabilistic objects.

In this limit analysis, two main issues arise. The first is the definition of a proper convergence of random functions to deterministic functions. The second is the development of technical tools of nonlinear homogenization type that allow one to treat the Brownian motion as a variable to be optimized. 

In order to properly study the homogenization process, we precisely define our functionals as
$$
F^\lambda_\e(v,\mathcal F_t,B_t)= \int_0^{+\infty} \EE\Big(\frac12 |\dot v(t)|^2+ V_{\rm per} \Big(\frac{\sqrt\e B_t+ v(t)}\e\Big)\Big)\, e^{-\lambda t}dt,
$$
explicitly depending also on a  filtration $\{\mathcal F_t\}_{t\ge 0}$ such that $v(t)$ is measurable with respect to $\mathcal F_t$ for every $t$, and  $\{B_t\}_{t\ge 0}$ a $d$-dimensional Brownian motion adapted to $\{\mathcal F_t\}_{t\ge 0}$. The convergence of random variables $v_\e$ to a deterministic function $v$ is then defined as the weak convergence of $v_\e$ towards a random function $\widehat v$, where $\EE(\widehat v)=v$, and the $\Gamma$-limit is defined with respect to this convergence  as $\e\to0$.
The main result is that this limit can be written in terms of a {\em deterministic  homogenized Lagrangian}  $f_{\rm hom}$ as
$$
F^\lambda_{\rm hom}(v)=\int_0^{+\infty} f_{
\rm hom}(\dot v(t))\, e^{-\lambda t}dt.
$$
The effective Lagrangian is defined by a variant of the usual asymptotic homogenization formula,
\begin{eqnarray}\label{fomma}\nonumber
&& f_{\rm hom}(\xi)=\liminf_{T\to+\infty} \frac1T\min\bigg\{\int_0^T \EE\big(\tfrac12|\dot w(t)|^2+V_{\rm per}(B_t+ w(t))\big)dt:
\\
&&\ \hskip6.7cm\  \EE(w(T, \cdot)-w(0,\cdot))=T\xi \Big\},
\end{eqnarray}
where the infimum is taken over all filtrations $\mathcal F_t$, Brownian motions $B_t$ adapted to $\mathcal F_t$, and functions $ w\in H^1((0,T);L^2(\Omega; \mathbb R^d))$, adapted to $\mathcal F_t$. This weak formulation is used in Stochastic control (see \cite{zbMATH06042622} and \cite[Chapter 2, Section 4]{zbMATH01325009}). 
The solutions (functions, filtrations, and Brownian motions) to the minimum problems in the definition of $f_{\rm hom}$ can then be used to construct random (almost) correctors, by concatenating not only the functions $v$ as usual,  but also the corresponding filtrations and Brownian motions. This concatenation is allowed by the possibility of choosing independent minimizing Brownian motions for the same minimum problem. 

The variational approach also allows us to prove that, when the periodic potential $V_{\rm per}$ is perturbed by adding some continuous non-negative function $V_0$ vanishing at infinity, the same homogenization limit is achieved.  Similar questions have been addressed for the inviscid first-order Hamilton--Jacobi equations both from the PDE theoretic viewpoint (see \cite{PLL-college,zbMATH07755395,zbMATH08099540,zbMATH06921166}) and the variational viewpoint \cite{MR5057984}. Finally, we remark that our approach extends to other Hamilton--Jacobi equations with more general Hamiltonians and non-constant viscosity (see Section 8).

\smallskip
Our article is organized as follows. Section 2 contains the definition (Definition 2.1)  of our notion of variational convergence for random functions and the statement of our main result, namely Theorem 2.2, of variational convergence on a finite time interval $(a,b)$. Indeed, since the exponential $e^{-\lambda t}$ is integrable on $(0,+\infty)$, the convergence results for the energies above can be obtained from the analogous result for energies localized on bounded time intervals. We address the latter setting until Section 6. First, Section 3 yields the lower limit. We construct a sequence of suitable correctors in Section 4 and use it for the proof of the upper limit in Section 5. We then address the case of a suitable perturbation of the periodic potential in Section 6. We  adapt our arguments to the infinite time interval; that is, to $(0,+\infty)$,  in Section 7 and conclude on the homogenization of the original problem \eqref{11}--\eqref{12}.  In Section 8 we comment on possible application of the convergence scheme to other settings. Our final Section 9 contains two classical properties of the Brownian motion that we include here for the convenience of the reader.

\section{Statement of the convergence theorem} 
Throughout the text, $(\Omega, \mathcal F, P)$ is a fixed probability space.
Given $\lambda>0$, $\mathcal F_t$ and $B_t$, we consider energies of the form
\begin{equation}\label{fellat}
F^\lambda_\e(v,\mathcal F_t,B_t)= \int_0^{+\infty} \EE\Big(\frac12 |\dot v(t)|^2+ V_{\rm per} \Big(\frac{\sqrt\e B_t+ v(t)}\e\Big)\Big)\, e^{-\lambda t}dt
\end{equation}
defined on functions $v=v(t,\omega)$ with $v(0,\omega)=x$. All our Brownian motions are standard, valued in $\mathbb R^d$, and start at $x=0$ for $t=0$.

The convergence results for the energies above can be obtained from the analogous result for energies localized on a bounded interval. With fixed $a,b$ we consider the functional
\begin{equation}\label{deffe}
F_\e(v,\mathcal F_t,B_t)= \int_a^{b} \EE\Big(\frac12 |\dot v(t)|^2+ V_{\rm per} \Big(\frac{\sqrt\e B_t+ v(t)}\e\Big)\Big)dt
\end{equation}
defined for $\mathcal F_t$ a filtration of $\mathcal F$, $B_t$ a Brownian motion adapted to $\mathcal F_t$, and $v\in H^1((a,b);L^2(\Omega))$ such that for every $t$ the function $v(t)$ is $\mathcal F_t$-measurable. In this notation, $\dot v$ denotes the time derivative, defined for almost all $t$ as the limit of the difference quotients in the strong topology in $L^2(\Omega)$.

We remark that the spaces $L^2((a,b);L^2(\Omega))$ and $H^1((a,b);L^2(\Omega))$ can be characterized equivalently by requiring that $\omega\mapsto \int_a^b|u(t,\omega)|^2dt$ and $\omega\mapsto \int_a^b|\dot u(t,\omega)|^2dt$ be integrable in $\Omega$.

\smallskip
We will consider the following notion of convergence. 

\begin{definition}[convergence to a deterministic variable]\label{defco}
Given a sequence $v_k$ with  $v_k\in H^1((a,b);L^2(\Omega;\mathbb R^d))$ and $v\in H^1((a,b);\mathbb R^d)$ we write $v_k\to v$ if 
for all subsequences $v_{k_j}$, there exist a subsequence, which we still denote by $v_{k_j}$, 
and $\widehat v\in H^1((a,b);L^2(\Omega;\mathbb R^d))$, possibly depending on the subsequence, such that 
$$
 v_{k_j}\wto \widehat v \hbox{ weakly in }H^1((a,b);L^2(\Omega;\mathbb R^d)),
\qquad
\EE(\widehat v(t))=v(t)\hbox{ for every }t.
$$
\end{definition}

We define $f_{\rm hom}\colon\mathbb R^d\to\mathbb R$ as
\begin{eqnarray}\label{fomd}
f_{\rm hom}(\xi)=\liminf_{T\to+\infty} \frac1T\inf_{\mathcal F_t,  B_t}\bigg\{\int_0^T \EE\Big(\frac12|\dot w(t)|^2+V_{\rm per}(B_t+ w(t))\Big)dt: w\in X_0^T(\xi)  \bigg\},
\end{eqnarray}where 
$X_0^T(\xi)$ 
is the set of all $w\in  H^1((a,b);L^2(\Omega;\mathbb R^d))$ such that $w(t)$ is $\mathcal F_t$-measurable for every $t$ and $\EE(w(T, \cdot)-w(0,\cdot))=T\xi $, and the minimum is also taken over all filtrations $\mathcal F_t$ and all $d$-dimensional Brownian motions $B_t$ adapted to $\mathcal F_t$.
Note that $X_0^T(\xi)$ depends on the filtration $\mathcal F_t$. Note moreover that,  using $w(t)= \xi t$ as a test function, 
\begin{equation}\label{13}
f_{\rm hom}(\xi)\le \frac12|\xi|^2+\max V_{\rm per}.
\end{equation}

Accordingly, we define 
\begin{equation}\label{Fhomdef}
F_{\rm hom}(v)=\int_a^b f_{\rm hom}(\dot v(t))dt,
\end{equation}
for $v\in H^1((a,b);\mathbb R^d)$. Given $\e_k\to 0$, our goal is to prove the $\Gamma$-convergence of $F_{\e_k}$ to $F_{\rm hom}$ in the following sense.

\begin{theorem}\label{main} Let $F_\e$ be given by \eqref{deffe}. Given $\e_k\to 0$, we have

\smallskip
{\rm(i)} For every $v\in H^1(a,b)$ and for every sequences $\mathcal F^k_t$,  $B^k_t$, and $v_k\to v$ such that 
$v_k(t)$ is $\mathcal F^k_t$-measurable, we have 
$$
F_{\rm hom} (v)\le\liminf_{k\to+\infty}F_{\e_k}(v_k, \mathcal F^k_t,B^k_t);
$$

\smallskip
{\rm(ii)} for every $v\in H^1(a,b)$ there exists sequences  $\mathcal F^k_t$,  $B^k_t$, and $v_k\to v$, such that 
$v_k(t)$ is $\mathcal F^k_t$-measurable,  such that
$$
F_{\rm hom} (v)\ge\limsup_{k\to+\infty}F_{\e_k}(v_k, \mathcal F^k_t,B^k_t).
$$
\end{theorem}
From this theorem, we will deduce the analogous result for the functionals \eqref{fellat}. 
We briefly explain how that convergence result allows us
to give an alternative characterization of $H_{\rm hom}$ as the Legendre transform of the homogenized Lagrangian $f_{\rm hom}$. The details will be given in Section \ref{stab-hl}.

Given $\lambda>0$ and $\e>0$, for all $p\in\mathbb R^d$ we consider the unique viscosity solution $U_{ p ,\e}$ of the Hamilton-Jacobi equation 
\begin{equation}
\lambda U_{ p ,\e}(x)-\frac\e2 \triangle U_{ p ,\e}(x)+H_{\rm per}\Big(\frac x\e,
 p +\nabla U_{ p ,\e}(x)\Big)=0,
\qquad x\in\mathbb R^d.
\end{equation}
From homogenization theory we know that
$U_{ p ,\e}(x)$ converge as $\e\to 0$ to the (constant) solution $U_ p $ of 
\begin{equation}\label{Homp}
\lambda U_{ p }(x)+H_{\rm hom}(
 p +\nabla U_{ p }(x))=0,
\qquad x\in\mathbb R^d.
\end{equation}

On the other hand, given a probability space $(\Omega, \mathcal F, P)$ and a $d$-dimensional Brownian motion $B_t$ on it,
the solution $U_{ p ,\e}$ can also be characterized as 
$$
U_{ p ,\e}(x)=\inf\Big\{
\int_0^{+\infty} \EE\Big(\frac12 |\dot v(t)|^2-\langle  p ,\dot v(t)\rangle+ V_{\rm per} \Big(\frac{\sqrt\e B_t+ v(t)}\e\Big)\Big)\, e^{-\lambda t}dt:  v(0)=x\Big\},
$$
where $\EE$ denotes the expectation with respect to $(\Omega, \mathcal F, P)$.
Here, the infimum is taken over all functions $v\colon(0,+\infty) \to L^2(\Omega;\mathbb R^d)$ in the corresponding weighted Sobolev space, satisfying the initial condition
and adapted to the filtration corresponding to the Brownian motion (see e.g.~\cite[Section I.8]{MR1484411} and \cite{MR2179357}).

Using the functionals $F^\lambda_\e$ in \eqref{fellat}, the solutions $U_{ p ,\e}(x)$ can be rewitten as
\begin{eqnarray}\label{Uetae}
U_{ p ,\e}(x)=\inf\Big\{ F^\lambda_\e(v,\mathcal F_t,B_t)-\int_0^{+\infty}\EE(\langle  p ,\dot v(t)\rangle) e^{-\lambda t}dt: v(0)=x\Big\},
\end{eqnarray}
where the integral will be treated separately as a continuous perturbation.
In Theorem \ref{main2} we will show that the limit of $F^\lambda_\e$ in the appropriate weighted Sobolev space can be written  as
$$
F^\lambda_{\rm hom}(v)=\int_0^{+\infty} f_{
\rm hom}(\dot v(t))\, e^{-\lambda t}dt.
$$

Since in this approach the integral in 
\eqref{Uetae} converges continuously to the corresponding deterministic term 
$\int_0^{+\infty}\langle  p ,\dot v(t)\rangle e^{-\lambda t}dt$,
the behaviour of the solutions $U_{ p ,\e}(x)$ can be characterized as converging to 
\begin{equation}
U_ p (x)=\min\Big\{ \int_0^{+\infty}
\big(f_{\rm hom}(\dot v(t))-\langle  p ,\dot v(t)\rangle \big)e^{-\lambda t}dt: v(0)=x\Big\}.
\end{equation}
This minimum value is independent of $x$, is equal to $\frac1 \lambda \min\{f_{\rm hom}(\xi)-\langle  p ,\xi\rangle\}$, and it is achieved by $v(t)=x+\xi t$, where $\xi
\in \mathbb R^d$ minimizes $f_{\rm hom}(\xi)-\langle  p ,\xi\rangle$.
This gives 
$$
\lambda U_ p =\min \{f_{\rm hom}(\xi)-\langle  p ,\xi\rangle\} =-f_{\rm hom}^*( p ),
$$
where $f_{\rm hom}^*$ denotes the Legendre transform of $f_{\rm hom}$. Since $U_ p $ is constant independent of $x$, from the previous equality it trivially solves 
$$
\lambda U_ p + f^*_{\rm hom}(p+\nabla U_ p )=\lambda U_ p + f^*_{\rm hom}( p )=0.
$$
Since $U_ p $ is also characterized as the solution of 
$$
\lambda U_ p + H_{\rm hom}(p+\nabla U_ p )=\lambda U_ p + H_{\rm hom}( p )=0,
$$
this shows that $f^*_{\rm hom}=H_{\rm hom}$ by the arbitrariness of $ p $.

\section{Lower bound}
This section is devoted to the proof of claim (i) in Theorem \ref{main}.
We will need the following lemma, which will be used to modify boundary values. Note that the functions implicitly depend on $\omega$.

\begin{lemma}[variation of boundary data]\label{lemmada} Let $V_k\colon(0,1)\times \mathbb R^d\to \mathbb R$ be Borel functions such that there exists $M>0$ with $|V_k(t,x)|\le M$ for all $k\in\mathbb N$, $t\in(0,1)$, and $x\in\mathbb R^d$. We define
\begin{equation}\label{degde}
G_k(z)=\int_0^1\EE\Big(\frac12|\dot z(t)|^2dt+V_k(t,z(t))\Big)dt
\end{equation}
for $z
\in H^1((0,1),L^2(\Omega;\mathbb R^d))$.
Let $\mathcal G^k_t$ be a sequence of filtrations and let $z_k\in H^1((0,1),L^2(\Omega;\mathbb R^d))$
be such that $z_k(t)$ is  $\mathcal G^k_t$-measurable for every $t\in(0,1)$. Assume that $z_k\wto \widehat z$ weakly in $H^1((0,1),L^2(\Omega;\mathbb R^d))$.
Then, given $\eta>0$ there exists a sequence $\widetilde z_k$ in $H^1((0,1),L^2(\Omega;\mathbb R^d))$ such that 

\smallskip
{\rm (i)} $\widetilde z_k\wto \widehat z$ weakly in $H^1((0,1),L^2(\Omega;\mathbb R^d))$;

\smallskip
{\rm (ii)} $\widetilde z_k(t)$ is  $\mathcal G^k_t$-measurable for every $t\in(0,1)$; 

\smallskip
{\rm (iii)} $G_{k} (\widetilde z_k) \le G_{k} (z_k)+\eta$ for $k$ large enough;

\smallskip
{\rm (iv)} $\EE(\widetilde z_k(0))=\EE(\widehat z(0))$ and $\EE(\widetilde z_k(1))=\EE(\widehat z(1))$.

\end{lemma}

\begin{proof} We first note that $\EE(z_k(0))\to\EE(\widehat z(0))$ (and similarly $\EE(z_k(1))\to\EE(\widehat z(1))$).
We set $\alpha_k=\EE(z_k(0))-\EE(\widehat z(0))$ and $\beta_k=\EE(z_k(1))-\EE(\widehat z(1))$. Let $\delta>0$ and let $\varphi\in C_c^\infty(0,1)$ with $0\le\varphi\le 1$ and $\varphi=1$ in $(\delta,1-\delta)$. We set
$$
\widetilde z_k(t)=\begin{cases} z_k(t)-(1-\varphi(t))\alpha_k & \hbox{ if } t\in(0,\delta)\\
 z_k(t) &\hbox{ if } t\in[\delta,1-\delta]\\
z_k(t)-(1-\varphi(t))\beta_k & \hbox{ if } t\in(1-\delta,1),
 \end{cases}
$$
so that
$$
\dot{\widetilde z}_k(t)=\begin{cases} \dot z_k(t) +\dot \varphi(t)\alpha_k & \hbox{ if } t\in(0,\delta)\\
\dot z_k(t) &\hbox{ if } t\in[\delta,1-\delta]\\
\dot z_k(t) +\dot \varphi(t)\beta_k & \hbox{ if } t\in(1-\delta,1)
 \end{cases}
$$
Properties (ii) and (iv) are true by construction, and (i) follows since  $\alpha_k,\beta_k\to 0$.
To prove (iii) we compute 
\begin{eqnarray*}
G_k (\widetilde z_k)&=& \int_0^{1} \EE\Big(\frac12 |\dot{\widetilde z}_k(t)|^2+ V_k(t, \widetilde z_k(t))\Big)dt
\\
&\le& \int_{\delta}^{1-\delta} \EE\Big(\frac12 |\dot{z}_k(t)|^2+  V_k(t, z_k(t))\Big)dt
\\
&&
+\int_0^{\delta} \EE\Big(\frac12 |\dot{z}_k(t)|^2+ \alpha_k\dot{z}_k(t)\dot \varphi(t) +\frac12\alpha_k^2 |\dot{\varphi}(t)|^2\Big)dt
\\
&&
+\int_{1-\delta}^1 \EE\Big(\frac12 |\dot{z}_k(t)|^2+ \beta_k\dot{z}_k(t)\dot \varphi(t) +\frac12\beta_k^2 |\dot{\varphi}(t)|^2\Big)dt
+2\delta \max |V_k|
\\
&\le& G_k(z_k)
+\int_0^{\delta} \EE\Big(\alpha_k\dot{z}_k(t)\dot \varphi(t) +\frac12\alpha_k^2 |\dot{\varphi}(t)|^2\Big)dt
\\
&&
+\int_{1-\delta}^1 \EE\Big(\beta_k\dot{z}_k(t)\dot \varphi(t) +\frac12\beta_k^2 |\dot{\varphi}(t)|^2\Big)dt
+4\delta M.
\end{eqnarray*}
Since $\alpha_k$ and $\beta_k\to 0$, claim (iii) follows by the arbitrariness of $\delta$.
\end{proof}

\begin{proof}[Proof of the lower bound \rm(claim {\rm(i)} of Theorem {\rm\ref{main}})]
We will use the blow-up method \cite{FM,BMS}, applied to the measures $\mu_k=g_k\mathcal L$, where
$$
g_k(t)=\EE\Big(\frac12 |\dot v_{k}(t)|^2+ V_{\rm per} \Big(\frac{\sqrt{\e_k} B^k_t+ v_k(t)}{\e_k}\Big)\Big) .
$$
It is not restrictive to suppose that there exists the limit
$$
\lim_{k\to+\infty}\mu_k((a,b))=\lim_{k\to+\infty}F_{\e_k}(v_k, \mathcal F^k_t,B^k_t)<+\infty.
$$
Hence, we can suppose that $\mu_k$ converge weakly* to  a bounded Radon measure $\mu$ on $(a,b)$. We let $\frac{d \mu}{dt}$ denote the density of its absolutely continuous part with respect to the Lebesgue measure $\mathcal L$.

We fix $t_0\in(a,b)$ a Lebesgue point for $\dot v$ such that
$$
\frac{d \mu}{dt}(t_0)=\lim_{\rho\to 0}\frac1{\rho}\mu((t_0,t_0+\rho)),
$$
and note that $\mathcal L$-almost every $t_0\in(a,b)$ satisfies this property. Since $\mu_k((t_0,t_0+\rho))\to \mu((t_0,t_0+\rho))$ except for at most a countable set of $\rho$, we can find $\rho_k\to 0$ such that
$$
\frac{d \mu}{dt}(t_0)= \lim_{k\to +\infty}\frac1{\rho_k}\mu_k((t_0,t_0+\rho_k));
$$
that is, 
\begin{equation}\label{limmu}
\frac{d \mu}{dt}(t_0)= \lim_{k\to +\infty}
\frac1{\rho_k}\int_{t_0}^{t_0+\rho_k}\EE\Big(\frac12 |\dot v_{k}(t)|^2+ V_{\rm per} \Big(\frac{\sqrt{\e_k} B^k_t+ v_k(t)}{\e_k}\Big)\Big)dt.
\end{equation}

With $\widehat v$ as in Definition \ref{defco}, we can also assume that $t_0$ is such that $(\widehat v(t_0+\rho)- \widehat v(t_0))/\rho$ converges strongly in $L^2(\Omega)$ as $\rho\to 0^+$ to $\widehat \xi\in L^2(\Omega;\mathbb R^d)$, the derivative of $\widehat v$ in $t_0$. We can then assume in addition that the parameter $\rho_k>0$ is also such that the scaled functions $z_k(t)=\frac1{\rho_k}(v_k(t_0+\rho_k t)- v_k(t_0))$ converge strongly  in $L^2((0,1),L^2(\Omega);\mathbb R^d)$ to $\widehat z(t)=\widehat \xi t$. 

In order to apply  Lemma \ref{lemmada} we now write
\begin{eqnarray*}
&&\hskip-.5cm\frac1{\rho_k}\int_{t_0}^{t_0+\rho_k}\EE\Big(\frac12 |\dot v_{k}(t)|^2+ V_{\rm per} \Big(\frac{\sqrt{\e_k} B^k_t+ v_k(t)}{\e_k}\Big)\Big)dt
\\
&&=\int_{0}^{1}\EE\Big(\frac12 |\dot v_{k}(t_0+\rho_kt)|^2+ V_{\rm per} \Big(\frac{\sqrt{\e_k} B^k_{t_0+\rho_k t}+ v_k(t_0+\rho_k t)}{\e_k}\Big)\Big)dt
\\
&&=\int_{0}^{1}\EE\Big(\frac12 |\dot z_{k}(t)|^2+ V_{\rm per} \Big(\frac{\sqrt{\e_k} B^k_{t_0+\rho_k t}+v_k(t_0)+\rho_ k z_k(t)} {\e_k}\Big)\Big)dt
\\
&&=\int_{0}^{1}\EE\Big(\frac12 |\dot z_{k}(t)|^2+ V_k (t,z_k(t))\Big)dt,
\end{eqnarray*}
where 
$$
V_k(t,x)= V_{\rm per} \Big(\frac{\sqrt{\e_k} B^k_{t_0+\rho_k t}+v_k(t_0)+\rho_ k x }
{\e_k}\Big).
$$

By applying Lemma \ref{lemmada} with the filtration $\mathcal G^k_t:=\mathcal F^k_{t_0+\rho_k t}$, we obtain $\widetilde z_k$ satisfying claims (i)--(iv) of that lemma, with $\widehat z(t)= \widehat \xi t$. As a consequence, by \eqref{limmu} we have
$$
\liminf_{k\to+\infty} \int_{0}^{1}\EE\Big(\frac12 |\dot z_{k}(t)|^2+ V_{\rm per} \Big(\frac{\sqrt{\e_k} B^k_{t_0+\rho_k t}+v_k(t_0)+\rho_ k \widetilde z_k(t)}{\e_k}\Big)\Big)dt\le \frac{d \mu}{dt}(t_0)+\eta.
$$

Changing variables 
$$
s=\frac{\rho_k}{\e_k}t,\qquad T_k= \frac{\rho_k}{\e_k},\qquad w_k(s)=\frac{\rho_k}{\e_k}\, \widetilde z_k\Big(\frac{\e_k}{\rho_k} s\Big) +\frac{v_k(t_0)}{\e_k},
$$
we can rewrite the integral as 
\begin{eqnarray}\label{limTk}\nonumber
&&\frac1{T_k}\int_{0}^{T_k}\EE\Big(\frac12 |\dot w_{k}(s)|^2+ V_{\rm per} 
\Big(\frac{B^k_{t_0+\e_k s}}{\sqrt{\e_k}}+ w_k(s)\Big)\Big)ds.
\\
&&=\frac1{T_k}\int_{0}^{T_k}\EE\Big(\frac12 |\dot w_{k}(s)|^2+ V_{\rm per} 
\Big(\frac{B^k_{t_0+\e_k s}-{B^k_{t_0}}}{\sqrt{\e_k}}+ \frac{{B^k_{t_0}}}{\sqrt{\e_k}}+ w_k(s)\Big)\Big)ds.
\end{eqnarray}

We then define 
$$
\widehat w_k(s):=w_k(s){+\frac{B^k_{t_0}}{\sqrt{\e_k}}},\qquad\qquad
\widehat B^k_s:=\frac{B^k_{t_0+\e_k s}- B^k_{t_0}}{\sqrt{\e_k}},
$$
 and  rewrite \eqref{limTk} as 
\begin{equation}\label{limTk1}
\frac1{T_k}\int_{0}^{T_k}\EE\Big(\frac12{ \big|\dot{\widehat w}_{k}(s)\big|^2}+ V_{\rm per} 
({\widehat B^k_s}+ 
\widehat w_k(s))\Big)ds.
\end{equation}
If we define the filtration $\widehat {\mathcal F}^k_s:= \mathcal F^k_{t_0+\e_k s}$, to which the Brownian motion  $\widehat B^k_s$ is adapted, then $(\widehat {\mathcal F}^k_s,\widehat B^k_s,\widehat w_k)$ is a test triplet for the problem in \eqref{fomd} for $T=T_k$, since $\EE(\widehat w_k(T_k))-\EE(\widehat w_k(0))=\EE(w_k(T_k)-w_k(0))= \EE(\widehat\xi) T_k= \dot v(t_0)T_k$.
We then obtain 
$$
\frac{d \mu}{dt}(t_0)+\eta\ge f_{\rm hom} (\dot v(t_0)),
$$
which then holds for $\mathcal L$-almost all $t_0\in(a,b)$, and finally,
\begin{eqnarray*}
\lim_{k\to+\infty}F_{\e_k}(v_k, \mathcal F^k_t,B^k_t)
&=&\lim_{k\to+\infty}\mu_k(a,b)\ge \mu(a,b)\\
&\ge& \int_a^b\frac{d \mu}{dt}(t)dt\ge \int_a^b f_{\rm hom} (\dot v(t))dt-\eta(b-a),
\end{eqnarray*}
which gives the lower bound by the arbitrariness of $\eta>0$.\end{proof}

\section{Construction of correctors}\label{corrector-sec} 
In order to prove an upper bound it is convenient, for a given $\xi$, to define $B_t$ and $w$ on the whole half-line $[0,+\infty)$ such that $\sqrt\e B_{t/\e}+ \e w(t/\e)$ is an (almost) corrector for $\xi t$. This is done by concatenating independent copies of an almost minimizer for $f_{\rm hom}$. 

\begin{proposition}[correctors]\label{correct}
With fixed $\xi\in\mathbb R^d$,  $\eta>0$, $K>0$, $T_0>0$ there exist $T>T_0$, $r>0$, a filtration $\mathcal F^\xi_t$, a function $w^\xi \in H^1((0,R),L^2(\Omega;\mathbb R^d))$ for all $R>0$, and a Brownian motion $B^\xi_t$, with both $w^\xi$ and $B^\xi_t$ adapted to $\mathcal F^\xi_t$,  such that for all $0\le s_1<s_2<+\infty$
\begin{eqnarray}\label{tre-b}\nonumber
&&\EE\Big(\int_{s_1}^{s_2} \Big(\frac12|\dot{w}{}^\xi(t)|^2+V_{\rm per}(B^\xi_t+ w^\xi(t))\Big)dt\Big)
\\
&\le&{
(s_2-s_1)\Big( f_{\rm hom}(\xi)+\eta +\big(3\frac{K}T+{2}\eta_K\big)\max|V_{\rm per}|\Big)}+ C(\xi, T, K, \eta),
\end{eqnarray}
where
$$
C(\xi,T,K,\eta):=
2T (f_{\rm hom}(\xi)+\eta) +2(4K+T(1+{2}\eta_K))\max|V_{\rm per}|\Big),\\
$$
and $\eta_K\to 0$ as $K\to+\infty$. Moreover, for all $j\in\mathbb N$, we have
\begin{equation}
\EE(w^\xi(jT+(j-1)K)) -\EE(w^\xi(0))= jT\xi .
\end{equation}
\end{proposition}

\begin{proof}
We choose $T>T_0$ such that 
$$
\frac1T\min\bigg\{\int_0^T \EE\Big(\frac12|\dot w(t)|^2+V_{\rm per}(B_t+ w(t))\Big)dt:\mathcal F_t, w\in X_0^T(\xi), B_t\bigg\}<  f_{\rm hom}(\xi)+\frac\eta2.
$$
Recall that $X_0^T(\xi)$ depends on the filtration $\mathcal F_t$. 
Therefore, there exist a filtration $\mathcal F_t$, a function $w\in X_0^T(\xi)$, and a Brownian motion $B_t$ adapted to $\mathcal F_t$, such that 
\begin{equation}\label{mineta}
\frac1T \int_0^T \EE\Big(\frac12|\dot w(t)|^2+V_{\rm per}(B_t+ w(t))\Big)dt< f_{\rm hom}(\xi)+\frac\eta2.
\end{equation}
The function $w$ is extended as a constant $w(t)(\omega)=w(T)(\omega)$ for $t>T$. 

In order to proceed inductively in the following, we will need independent copies of $B_t$ and $w$.
To that end, we note that, up to enlarging $\Omega$, we can find a sequence of mutually independent filtrations $\mathcal F^j_t$, a sequence of Brownian motions $B^j_t$, adapted to $\mathcal F^j_t$, and corresponding $w^j$ defined on the same $(0,T)$ belonging to the corresponding $X_0^{j,T}(\xi)$, equivalent to $B_t$ and $w$; that is, such that
\begin{eqnarray}\label{equibw}
&\displaystyle \int_0^T \EE\big(|\dot w^j(t)|^2\big)dt=
\int_0^T \EE\big(|\dot w(t)|^2\big)dt,&
\\ \label{equibw-2}
 &\displaystyle \int_0^T \EE\big(V_{\rm per}(B^j_t+ w^j(t))\big)dt
= \int_0^T \EE\big(V_{\rm per}(B_t+ w(t))\big)dt.&
\end{eqnarray}
Indeed, we can replace $\Omega$ with $\Omega^\mathbb N$ with the product probability. After denoting $\pi_j$ the projection on the $j$-th factor, we define the filtrations 
$$
\mathcal F^j_t=\{ \pi_j^{-1}(A): A\in \mathcal F_t\}.
$$
Note that, given distinct indices $j_1,\ldots, j_m\in\mathbb N$, arbitrary times  $t_1,\ldots, t_m\in[0,T]$, and  a family $E_i\in \mathcal F^{j_i}_{t_i}$ we have
$$
P(E_1\cap\ldots\cap E_m)=\prod_{i=1}^mP(E_i);
$$
that is, $\mathcal F^j_t$ are independent. We then set
$$
B^j_t=B_t \circ \pi_j,\qquad w^j(t)=w(t)\circ \pi_j,
$$
both of which are adapted to $\mathcal F^j_t$. 
Note that \begin{equation}\label{cowj}
w^j(t)=w^j(T)\hbox{ for all $j$ and for all $t>T$}.
\end{equation}
Note that for almost every $t\ge 0$ the laws of $\dot w^j(t)$ and $B^j_t+w^j(t)$ coincide with the corresponding laws of $\dot w(t)$ and $B_t+w(t)$, from which we obtain
equalities \eqref{equibw} and \eqref{equibw-2}.

For all $j\in\mathbb N\cup\infty$ we also introduce the filtration
$$
\widetilde{\mathcal F}^j_t:=\sigma\Big(\bigcup_{i=1}^j\mathcal F^i_t\Big).
$$

\smallskip
To define $B^\xi_t$ and $w^\xi$ on the half line $(0,+\infty)$ we proceed by defining inductively Brownian motions $\widetilde B^j_t$ and functions $\widetilde w_j$, and then setting $B^\xi_t=\widetilde B^j_t$ and $w^\xi=\widetilde w_j$ on intervals $(\tau^\xi_{j-1},\tau^\xi_{j})$ between two suitable stopping times.

Let $\widetilde \tau_0=0$. We fix $r>0$ such that the modulus of continuity $\lambda$ for $V_{\rm per}$ satisfies $\lambda(r)<\eta/2$.
We consider the {\em stopping time} $\tau_1$ with respect to the filtration $\widetilde{\mathcal F}^2_t$ defined as the minimal $t\ge T$ such that
\begin{equation}\label{r-1}
{\rm dist}(B^1_t+w^1(T)-w^2(0),\mathbb Z^d)\le r.
\end{equation}
Note that in the scalar case $d=1$ this stopping time can be replaced by the one defined by the condition
\begin{equation}\label{simpz-1}
B^1_t+w^1(T)-w^2(0)\in\mathbb Z,
\end{equation}
avoiding the error term $r$. The same remark holds for the stopping times $\tau_j$ defined in the following.

For fixed $K>0$ there exits $\eta_K>0$, with $\eta_K\to 0$ as $K\to+\infty$, such that, having set $\Omega^1_K=\{\tau_1>T+ K\}$, we have  $P(\Omega^1_K)<\eta_K$  {(see Proposition \ref{apbrown})}. In other words we have $\tau_1\le T+ K$ on $\Omega\setminus\Omega^1_K$.

Let $\widetilde \tau_1=\tau_1\wedge (T+K)$, so that $T\le \widetilde \tau_1\le T+K$.
We set 
$$
 \widetilde B^1_t=B^1_t,\qquad\widetilde w^1(t) =w^1(t),
$$
recalling that $w^1(t)=w^1(T)$ for $t\ge T$, so that 
\begin{equation}\EE(\widetilde w^1(T))=\xi T+ \EE(w^1(0)).
\end{equation} 
Note that
$$
\int_0^{\widetilde \tau_1} \Big(\frac12|\dot{\widetilde w}^1(t)|^2+V_{\rm per}(\widetilde B^1_t+ {\widetilde w}^1(t))\Big)dt
=\int_0^{\widetilde \tau_1} \Big(\frac12|\dot{w}^1(t)|^2+V_{\rm per}(B^1_t+ {w}^1(t))\Big)dt
$$
$$
=\int_0^{T} \Big(\frac12|\dot{w}^1(t)|^2+V_{\rm per}(B^1_t+ {w}^1(t))\Big)dt
+\int_T^{\widetilde \tau_1} \Big(\frac12|\dot{w}^1(t)|^2+V_{\rm per}(B^1_t+ {w}^1(t))\Big)dt
$$
Hence, by \eqref{mineta},  \eqref{equibw} and \eqref{cowj}, we have
\begin{eqnarray*}
&&\hskip-1cm\EE\Big(\int_0^{\widetilde \tau_1} \Big(\frac12|\dot{\widetilde w}^1(t)|^2+V_{\rm per}(\widetilde B^1_t+ {\widetilde w}^1(t))\Big)dt\Big)
\\
&&= \int_0^T \EE\Big(\frac12|\dot w(t)|^2+V_{\rm per}(B_t+ w(t))\Big)dt
+\EE
\Big(\int_T^{\widetilde \tau_1} V_{\rm per}(B^1_t+ {w}^1(T))dt
\Big)
\\
&&\le \int_0^T \EE\Big(\frac12|\dot w(t)|^2+V_{\rm per}(B_t+ w(t))\Big)dt
+\EE
\Big(\int_T^{\widetilde \tau_1} |V_{\rm per}(B^1_t+ {w}^1(T))|dt
\Big)
\\
&&\le 
T\Big( f_{\rm hom}(\xi)+\frac\eta2\Big) +K\max V_{\rm per}.
\end{eqnarray*}

\smallskip
In order to illustrate the induction argument we preliminarily deal with $j=2$. We define $\widetilde w^2(t)$ and $\widetilde B^2_t$ as follows. We set 
$$
\widetilde w^2(t)=  \begin{cases}\widetilde w^1(t) &\hbox{ for } 0\le t\le \widetilde \tau_{1}\\
\widetilde w^1(\widetilde\tau_{1})+ w^2(t-\widetilde\tau_{1})-  w^2(0)
&\hbox{ for }t> \widetilde \tau_{1}.\end{cases}
$$
Observe that $\widetilde w^2\in H^1((0,R);L^2(\Omega;\mathbb R^d))$ for every $R>0$, and that for every $t>0$ the function $w^2(t)$ is $\widetilde{\mathcal F}^2_t$-measurable. Note moreover that, since $\widetilde w^1(\widetilde\tau_{1})=w^1(T)$, 
 \begin{equation}
 \widetilde w^2(t)=\widetilde w^2(T+\widetilde\tau_{1})=
w^1(T)+ w^2(T)-  w^2(0)\quad\hbox{ if }t\ge T+\widetilde\tau_{1};
 \end{equation}
in particular, 
 \begin{equation}
 \widetilde w^2(t)=\widetilde w^2(2T+K)=w^1(T)+ w^2(T)-  w^2(0)\quad\hbox{ if }t\ge 2T+K
 \end{equation}
 since 
 $\widetilde w^1(\widetilde\tau_{1})=w^1(T)$ and $w^2(2T+K-\widetilde\tau_{1})=w^2(T)$, recalling that $2T+K-\widetilde\tau_{1} \ge T$.
 
We also set
$$
\widetilde B^2_t = \begin{cases}\widetilde B^1_t  &\hbox{ for } 0\le t\le \widetilde \tau_{1} \\
 B^2_{t-\widetilde \tau_{1} } + \widetilde B^1_{\widetilde \tau_{1} }  &\hbox{ for }t> \widetilde \tau_{1} .\end{cases}
$$
By the strong Markov property this formula defines a Brownian motion on $[0,+\infty)$ adapted to the filtration $\widetilde{\mathcal F}^2_t$ (see Proposition \ref{concat}). This argument will be used iteratively in the following.

We define the stopping time $\tau_{2} $ with respect to the filtration $\widetilde{\mathcal F}^3_t$  as the minimal $t\ge 2T+K$ such that
$$
{\rm dist}(\widetilde B^2_t  + \widetilde w^2(t) - w^3(0) ,\mathbb Z^d)\le r,
$$
or, equivalently, such that
$$
{\rm dist}(\widetilde B^2_t  + \widetilde w^2(\widetilde\tau_{1} +T) - w^3(0) ,\mathbb Z^d)\le r,
$$
recalling that $w^2(t)=w^2(T)$ for $t\ge T$.
We then set $\widetilde\tau_{2} = \tau_{2} \wedge( 2T+ 2K),
$
so that $2T+K\le \widetilde \tau_2\le 2T+2K$.
Note that, having set $\Omega^2_K=\{\widetilde\tau_{2}\neq \tau_{2}\}$, we still have $P(\Omega^2_K)\le\eta_K$. 

Using the definitions of $\widetilde B^2_t$ and ${\widetilde w}^2$ we have
\begin{eqnarray*}
&&\int_0^{\widetilde \tau_2} \Big(\frac12|\dot{\widetilde w}^2(t)|^2+V_{\rm per}(\widetilde B^2_t+ {\widetilde w}^2(t))\Big)dt
\\&&
=\int_0^{\widetilde \tau_1} \Big(\frac12|\dot{\widetilde w}^2(t)|^2+V_{\rm per}(\widetilde B^2_t+ {\widetilde w}^2(t))\Big)dt +
\int_{\widetilde \tau_1}^{\widetilde \tau_2} \Big(\frac12|\dot{\widetilde w}^2(t)|^2+V_{\rm per}(\widetilde B^2_t+ {\widetilde w}^2(t))\Big)dt
\\
&&
=\int_0^{\widetilde \tau_1} \Big(\frac12|\dot{\widetilde w}^1(t)|^2+V_{\rm per}(\widetilde B^1_t+ {\widetilde w}^1(t))\Big)dt 
\\
&&\quad+
\int_{\widetilde \tau_1}^{\widetilde \tau_1+T} \Big(\frac12|\dot{w}^2(t-\widetilde \tau_1)|^2+V_{\rm per}(B^2_{t-\widetilde \tau_{1}}+ \widetilde B^1_{\widetilde \tau_{1}}+ \widetilde w^1(\widetilde\tau_{1})+ w^2(t-\widetilde\tau_{1})-  w^2(0))\Big)dt
\\
&&\quad
+\int_{\widetilde \tau_1+T}^{\widetilde \tau_2} \Big(\frac12|\dot{\widetilde w}^2(t)|^2+V_{\rm per}(\widetilde B^2_t+ {\widetilde w}^2(t))\Big)dt
\\
&&
=\int_0^{\widetilde \tau_1} \Big(\frac12|\dot{\widetilde w}^1(t)|^2+V_{\rm per}(\widetilde B^1_t+ {\widetilde w}^1(t))\Big)dt 
\\
&&\quad
+
\int_{0}^{T} \Big(\frac12|\dot{w}^2(t)|^2+V_{\rm per}(B^2_{t}+ \widetilde B^1_{\widetilde \tau_{1}}+ \widetilde w^1(\widetilde\tau_{1})+ w^2(t)-  w^2(0))\Big)dt
\\
&&\quad
+\int_{\widetilde \tau_1+T}^{\widetilde \tau_2} V_{\rm per}(\widetilde B^2_t+ {\widetilde w}^2(t))dt.
\end{eqnarray*}

On $\Omega\setminus\Omega^1_K$ we then have 
\begin{eqnarray*}
\int_{0}^{T} \Big(\frac12|\dot{w}^2(t)|^2+V_{\rm per}(B^2_{t}+ \widetilde B^1_{\widetilde \tau_{1}}+ \widetilde w^1(\widetilde\tau_{1})+ w^2(t)-  w^2(0))\Big)dt
\\
\le \int_{0}^{T} \Big(\frac12|\dot{w}^2(t)|^2+V_{\rm per}(B^2_{t}+w^2(t))\Big)dt +\lambda(r) T.
\end{eqnarray*}
Note that in dimension $1$ we can omit this error term.

On $\Omega^1_K$
\begin{eqnarray*}
&&
\int_{0}^{T} \Big(\frac12|\dot{w}^2(t)|^2+V_{\rm per}(B^2_{t}+ \widetilde B^1_{\widetilde \tau_{1}}+ \widetilde w^1(\widetilde\tau_{1})+ w^2(t)-  w^2(0))\Big)dt
\\&&=
\int_{0}^{T} \Big(\frac12|\dot{w}^2(t)|^2+V_{\rm per}(B^2_{t}+w^2(t))\Big)dt\\
&&\quad+\int_{0}^{T} \Big(V_{\rm per}(B^2_{t}+ \widetilde B^1_{\widetilde \tau_{1}}+ \widetilde w^1(\widetilde\tau_{1})+ w^2(t)-  w^2(0))-V_{\rm per}(B^2_{t}+w^2(t))\Big)dt
\\&&\le 
\int_{0}^{T} \Big(\frac12|\dot{w}^2(t)|^2+V_{\rm per}(B^2_{t}+w^2(t))\Big)dt+2T \max|V_{\rm per}|.
\end{eqnarray*}

Hence, 
\begin{eqnarray*}
&&
\EE\Big(\int_0^{\widetilde \tau_2} \Big(\frac12|\dot{\widetilde w}^2(t)|^2+V_{\rm per}(\widetilde B^2_t+ {\widetilde w}^2(t))\Big)dt\Big)
\\&&
\le 
2T (f_{\rm hom}(\xi)+\eta) +4K\max V_{\rm per} + 2T\eta_K\max|V_{\rm per}|,
\end{eqnarray*}
recalling that $\lambda(r)<\eta/2$.

\bigskip 
We now prove by induction that for all $j\ge 2$, there exist $\widetilde w^{j}$, $\widetilde B^{j}_t$, and $\widetilde \tau^{j}$ with the following properties:

\smallskip
(i) $\widetilde w^{j}\in H^1((0,R);L^2(\Omega;\mathbb R^d))$ for every $R>0$, $w^{j}(t)$ is $\widetilde{\mathcal F}^{j}_t$-measurable  for every $t>0$, 
 \begin{equation}\label{ii2}
 \widetilde w^j(t)=\widetilde w^j(  j T+(j-1)K) \hbox{ for } t\ge j T+(j-1)K,
 \end{equation}
and
 \begin{equation}\label{312}
 \widetilde w^j(jT+(j-1)K)=\sum_{\ell=1}^j(w^\ell(T)-  w^\ell(0)) +w^1(0);
 \end{equation}
 
\smallskip
(ii) $\widetilde B^{j}_t$ is Brownian motion adapted to the filtration $\widetilde{\mathcal F}^{j}_t$;

\smallskip
(iii) $\widetilde \tau_{j}$ is a stopping time with respect to the filtration $\widetilde{\mathcal F}^{j+1}_t$ such that 
 \begin{equation}\label{ii1}
jT+(j-1)K\le \widetilde \tau_{j}\le j T+ jK,
 \end{equation}
and
\begin{equation}\label{r-j}
{\rm dist}\big(\widetilde B^{j}_{\tau_{j} } + \widetilde w^{j}(\widetilde \tau_{j} ) - w^{j+1}(0) ,\mathbb Z^d\big)\le r
\end{equation}
except for a set $\Omega^{j}_K$ with $P(\Omega^{j}_K)< \eta_K$;

\smallskip
(iv) 
we have
\begin{eqnarray}\label{tre}\nonumber
&&\EE\Big(\int_{\widetilde \tau_{j-1}}^{\widetilde \tau_{j}} \Big(\frac12|\dot{\widetilde w}^{j}(t)|^2+V_{\rm per}(\widetilde B^{j}_t+ {\widetilde w}^{j}(t))\Big)dt\Big)
\\
&\le& 
T (f_{\rm hom}(\xi)+\eta) +2K\max V_{\rm per} +2T\eta_K\max|V_{\rm per}|.
\end{eqnarray}

\medskip The construction above guarantees that these properties hold for $j=2$.
Suppose that we have constructed $\widetilde w^{j}$, $\widetilde B^{j}_t$ and $\widetilde \tau^{j}$ satisfying (i)--(iv) above.
We now define  $\widetilde w^{j+1}$, $\widetilde B^{j+1}_t$ and $\widetilde \tau^{j+1}$ satisfying (i)--(iv) above with $j+1$ in the place of $j$.

\smallskip
(i) For every $t>0$ we define $\widetilde w^{j+1}(t) $ as
$$
\widetilde w^{j+1}(t) =  \begin{cases}\widetilde w^{j}(t)  &\hbox{ for } 0\le t\le \widetilde \tau_{j} \\
\widetilde w^{j}(\widetilde\tau_{j} ) + w^{j+1}(t-\widetilde\tau_{j} ) -  w^{j+1}(0) 
&\hbox{ for }t> \widetilde \tau_{j} .\end{cases}
$$
We observe that $\widetilde w^{j+1}\in H^1((0,R);L^2(\Omega;\mathbb R^d))$ for every $R>0$, and that for every $t>0$ the function $\omega\mapsto w^{j+1}(t) $ is $\widetilde{\mathcal F}^{j+1}_t$-measurable. Moreover, by \eqref{ii1} and \eqref{ii2},
 \begin{eqnarray*}
 \widetilde w^{j+1}((j+1)T+jK)&=& \widetilde w^{j}(\widetilde\tau_{j})+ w^{j+1}((j+1)T+jK-\widetilde\tau_{j})-  w^{j+1}(0)
\\
&=&
\widetilde w^{j}(jT+(j-1)K)+ w^{j+1}(T)-  w^{j+1}(0)
\\
&=&
 \sum_{\ell=1}^j(w^\ell(T)-  w^\ell(0))+w^1(0) + w^{j+1}(T)-  w^{j+1}(0),
 \end{eqnarray*}
and claim (i) is proved.

\smallskip
(ii) We set
$$
\widetilde B^{j+1}_t = \begin{cases}\widetilde B^{j}_t  &\hbox{ for } 0\le t\le \widetilde \tau_{j} \\
 B^{j+1}_{t-\widetilde \tau_{j} } + \widetilde B^{j}_{\widetilde \tau_{j} }  &\hbox{ for }t> \widetilde \tau_{j} .\end{cases}
$$
By the strong Markov property this formula defines a Brownian motion on $[0,+\infty)$ adapted to the filtration $\widetilde{\mathcal F}^{j+1}_t$.

\smallskip
(iii) We define the stopping time $\tau_{j+1} $ with respect to the filtration $\widetilde{\mathcal F}^{j+2}_t$  as the minimal $t\ge (j+1)T+jK$ such that
$$
{\rm dist}\big(\widetilde B^{j+1}_t+ \widetilde w^{j+1}(t)- w^{j+2}(0),\mathbb Z^d\big)\le r,
$$
or, equivalently, such that
$$
{\rm dist}\big(\widetilde B^{j+1}_t+ \widetilde w^{j+1}(\widetilde\tau_{j}+T)- w^{j+2}(0),\mathbb Z^d\big)\le r,
$$
recalling that $w^{j+1}(t)=w^{j+1}(T)$ for $t\ge T$.
Then, we set $$\widetilde\tau_{j+1} = \tau_{j+1} \wedge( (j+1)T+(j+1)K).$$ 
Note that, having set $\Omega^{j+1}_K=\{\widetilde\tau_{j+1}\neq\tau_{j+1}\}$, we still have $P(\Omega^{j+1}_K)<\eta_K$. 

\smallskip
(iv) We finally estimate
$$
\EE\Big(\int_{\widetilde \tau_{j}}^{\widetilde \tau_{j+1}} \Big(\frac12|\dot{\widetilde w}^{j+1}(t)|^2+V_{\rm per}(\widetilde B^{j+1}_t+ {\widetilde w}^{j+1}(t))\Big)dt\Big).
$$
To that end, we note that
\begin{eqnarray*}\nonumber
&&\int_{\widetilde \tau_{j}}^{\widetilde \tau_{j+1}} \Big(\frac12|\dot{\widetilde w}^{j+1}(t)|^2+V_{\rm per}(\widetilde B^{j+1}_t+ {\widetilde w}^{j+1}(t))\Big)dt\\ \nonumber
&=&\int_{\widetilde \tau_{j}}^{\widetilde \tau_{j}+T}\Big(\frac12|\dot{\widetilde w}^{j+1}(t)|^2+V_{\rm per}(\widetilde B^{j+1}_t+ {\widetilde w}^{j+1}(t))\Big)dt\\
&&+\int_{\widetilde \tau_{j}+T}^{\widetilde \tau_{j+1}}\Big(\frac12|\dot{\widetilde w}^{j+1}(t)|^2+V_{\rm per}(\widetilde B^{j+1}_t+ {\widetilde w}^{j+1}(t))\Big)dt,
\end{eqnarray*}
and that
\begin{eqnarray*}
&&
\int_{\widetilde \tau_{j}}^{\widetilde \tau_{j}+T}\Big(\frac12|\dot{\widetilde w}^{j+1}(t)|^2+V_{\rm per}(\widetilde B^{j+1}_t+ {\widetilde w}^{j+1}(t))\Big)dt
\\
&&=\int_{\widetilde \tau_{j}}^{\widetilde \tau_{j}+T}\Big(\frac12|\dot w^j(t-\widetilde\tau_{j})|^2\\
&&\qquad +V_{\rm per}(
B^{j+1}_{t-\widetilde \tau_{j}}+ \widetilde B^{j}_{\widetilde \tau_{j}}+ \widetilde w^{j}(\widetilde\tau_{j})+ w^{j+1}(t-\widetilde\tau_{j})-  w^{j+1}(0)\Big)dt
\\
&&
=\int_{0}^{T}\Big(\frac12|\dot w^{j+1}(t)|^2+V_{\rm per}(
B^{j+1}_{t}+ \widetilde B^{j}_{\widetilde \tau_{j}}+ \widetilde w^{j}(\widetilde\tau_{j})+ w^{j+1}(t)-  w^{j+1}(0)\Big)dt.
\end{eqnarray*}

On $\Omega\setminus \Omega^{j}_K$ by the inductive hypothesis (iii) we have
$$
\int_{0}^{T}\Big(\frac12|\dot w^{j+1}(t)|^2+V_{\rm per}(
B^{j+1}_{t}+ \widetilde B^{j}_{\widetilde \tau_{j}}+ \widetilde w^{j}(\widetilde\tau_{j})+ w^{j+1}(t)-  w^{j+1}(0)\Big)dt.
$$
$$
\le \int_{0}^{T}\Big(\frac12|\dot w^{j+1}(t)|^2+V_{\rm per}(B^{j+1}_{t}+  w^{j+1}(t))\Big)dt+\lambda(r)T,
$$
while on $\Omega^{j}_K$ we write 
\begin{eqnarray*}
&&\int_{0}^{T}\Big(\frac12|\dot w^{j+1}(t)|^2+V_{\rm per}(
B^{j+1}_{t}+ \widetilde B^{j}_{\widetilde \tau_{j}}+ \widetilde w^{j}(\widetilde\tau_{j})+ w^{j+1}(t)-  w^{j+1}(0)\Big)dt.
\\
&&
=\int_{0}^{T}\Big(\frac12|\dot w^{j+1}(t)|^2+V_{\rm per}(B^{j+1}_{t}+  w^{j+1}(t))\Big)dt
\\&&
+\int_{0}^{T}\Big(V_{\rm per}(
B^{j+1}_{t}+ \widetilde B^{j}_{\widetilde \tau_{j}}+ \widetilde w^{j}(\widetilde\tau_{j})+ w^{j+1}(t)-  w^{j+1}(0)\Big)-V_{\rm per}(B^{j+1}_{t}+  w^{j+1}(t))\Big)dt.
\end{eqnarray*}
Hence, recalling that $\lambda(r)<\eta/2$,
\begin{eqnarray}\label{stima-1}\nonumber
&&\EE\Big(\int_{\widetilde \tau_{j}}^{\widetilde \tau_{j}+T}\Big(\frac12|\dot{\widetilde w}^{j+1}(t)|^2+V_{\rm per}(\widetilde B^{j+1}_t+ {\widetilde w}^{j+1}(t))\Big)dt\Big)\\
&
\le&  T (f_{\rm hom}(\xi)+\eta) + T 2\eta_K \max|V_{\rm per}|.
\end{eqnarray}

Furthermore, we have 
\begin{eqnarray}\label{stima-2}\nonumber
\int_{\widetilde \tau_{j}+T}^{\widetilde \tau_{j+1}}\Big(\frac12|\dot{\widetilde w}^{j+1}(t)|^2+V_{\rm per}(\widetilde B^{j+1}_t+ {\widetilde w}^{j+1}(t))\Big)dt
&=&\int_{\widetilde \tau_{j}+T}^{\widetilde \tau_{j+1}}V_{\rm per}(\widetilde B^{j+1}_t+ {\widetilde w}^{j+1}(t))dt
\\
& \le& 2 K\max|V_{\rm per}|.
\end{eqnarray}
From \eqref{stima-1} and \eqref{stima-2} we obtain claim (iv) with $j$ replaced by $j+1$.

\medskip

We finally define $\mathcal F^\xi_t=\widetilde{\mathcal F}^\infty_t$, $B^\xi_t$ and $w^\xi$ on $(0,+\infty)$ as 
\begin{equation}
B^\xi_t= \widetilde B^{j}_t,\qquad w^\xi(t)= \widetilde w^{j} (t)\hbox{ for } t\le jT.
\end{equation}
Since $\widetilde \tau_j\ge jT$ and  we have $\widetilde B^{j}_t= \widetilde B^k_t$ and $\widetilde w^{j}=\widetilde w^k$ on $[0,\widetilde \tau_j]$ for $k\ge j$, this is a good definition. We observe that $B^\xi_t$ is a Brownian motion adapted to the filtration $\mathcal F^\xi_t$, and   $w^\xi\in H^1((0,R);L^2(\Omega;\mathbb R^d))$ for every $R>0$, and that for every $t>0$ the function $w^\xi(t)$ is $\mathcal F^\xi_t$-measurable. 
Eventually, \eqref{tre-b}  follows from \eqref{ii1} and \eqref{tre} by choosing $r$ small enough, and estimating
the integral on $(s_1,s_2)$ with the integral between $(\widetilde\tau_{j_1}, \widetilde\tau_{j_2})$, where 
$\widetilde\tau_{j_1}$ is the largest stopping time not greater than $s_1$ and 
$\widetilde\tau_{j_2}$ is the smallest stopping time not smaller than $s_2$.
\end{proof}

\begin{remark}[correctors in the one-dimensional case]\rm If $d=1$, as remarked in the proof, taking into account condition \eqref{simpz-1} and its iterations at every $j$, stopping times $\widetilde \tau_{j}$ can be assumed  to satisfy
$$
jT+(j-1)K\le \widetilde \tau_{j}\le j T+ jK,
$$
and
$$
B^\xi_{\tau_{j} } + w^\xi(\widetilde \tau_{j} ) - w^{j+1}(0) \in \mathbb Z^d
$$
except for a set $\Omega^{j}_K$ with $P(\Omega^{j}_K)< \eta_K$, avoiding the error term $r$.
\end{remark}

\section{Upper bound}\label{up}
We now prove claim (ii) of Theorem \ref{main}. 
 We begin with the case of an affine target function; that is, $v(t):=\xi t+x$, for  $t\in(a,b)$, with $x,\xi\in\mathbb R^d$. 

With fixed $\eta>0$ and $K$ we define  the correctors $\mathcal F^\xi_t$, $B^\xi_t$ and $w^\xi$ as in Proposition \ref{correct}.
We then define $v_\e$ and a new Brownian motion $B^\e_t$ on $(0,+\infty)$ by
\begin{equation}\label{veh}
v_\e(t)=\e\, w^\xi\Big(\frac{t}\e\Big)+\e\Big\lfloor \frac{x}\e\Big\rfloor,\quad\hbox{and}\quad B^\e_t= \sqrt\e B^\xi_{\frac{t}\e},
\end{equation}
where for every $y\in\mathbb R^d$ the symbol $\lfloor y\rfloor$ denotes the vector whose components are the integer parts of the components of $y$.

We observe that $B^\e_t$ is a Brownian motion adapted to the new filtration $\mathcal F^\e_t=\mathcal F^\xi_{t/\e}$, $v_\e\in H^1((a,b);L^2(\Omega;\mathbb R^d))$, and $v_\e(t)$ is $\mathcal F^\e_t$-measurable for every $t\in(a,b)$.

\begin{proposition}\label{pro1} 
There exists a constant $M>0$, independent of $K$ and $T$, such that 
$$
\EE\Big(\int_a^b|\dot v_\e(t)|^2dt\Big)\le M (b-a)(|\xi|^2+1)
$$
for every $0\le a<b<+\infty$ and $0<\e<(b-a)/(3T+3K)$. 
\end{proposition}

\begin{proof}
With fixed  $0\le a<b<+\infty$, and  $0<\e<(b-a)/(3T+3K)$, let $j_a$ be the largest integer such that $\e (j_aT+(j_a-1)K)\le a$, and $j_b$ be the
 least integer such that $\e (j_bT+(j_b-1)K)\ge b$. 
  Then, after a change of variables, 
 we have
\begin{eqnarray*}
\EE\Big(\int_a^b|\dot v_\e(t)|^2dt\Big)&\le&\e
\EE\,\bigg({\sum_{\ell=j_a}^{j_b-1}}\int_{\widetilde\tau_{\ell}}^{\widetilde\tau_{\ell}+T}|\dot w^\xi(t)|^2dt \bigg)
=\e
\EE\bigg({\sum_{\ell=j_a}^{j_b-1}}\int_{0}^{T}|\dot w^\ell(t)|^2dt \bigg)\\
&
=&\e
(j_b-j_a)\EE\,\bigg(\int_{0}^{T}|\dot w(t)|^2dt \bigg)\le\Big(\frac{b-a}{T+K}+2\e\Big)\EE\Big(\int_{0}^{T}|\dot w(t)|^2dt\Big)\\
&\le& 
2\Big(\frac{b-a}{T+K}+2\e\Big)\Big(\int_{0}^{T}\EE\Big(\frac12|\dot w(t)|^2+V_{\rm per} (B_t+w(t))\Big)dt
+ T\max|V_{\rm per}|\Big)
\\
&\le& 
4(b-a) (|\xi|^2+\eta+ 2\max|V_{\rm per}|),
\end{eqnarray*}
using that $f_{\rm hom}(\xi)\le \frac12 |\xi|^2+\max|V_{\rm per}|$ and that $\e\le (b-a)/(3T+3K)$.
The claim is obtained by sending $\eta$ to $0$, and choosing e.g.~$M=4(1+2\max|V_{\rm per}|)$.
\end{proof}

\begin{proposition}\label{pro2} Let $0\le a<b<+\infty$.
For all sequences $\e_k\to 0$, the sequence $v_{\e_k}$ in \eqref{veh} converges in $(a,b)$ in the sense of Definition {\rm\ref{defco}} to the function given by $v_{\xi,T,K}$ given by 
$$
v_{\xi,T,K}(t)= x+\frac{T}{T+K} t \xi.
$$
\end{proposition} 

\begin{proof}
Since $\e\lfloor \frac{x}\e\rfloor$ converges to $x$ in $\mathbb R^d$, it suffices to consider the case $x=0$.
By Proposition \ref{pro1}, up to subsequences, we can suppose that $v_{\e_k}\wto \widehat v$ weakly in $H^1((a,b);L^2(\Omega;\mathbb R^d))$.
To conclude we have only to prove that $\EE(\widehat v(t)) = v_{\xi,T,K}(t)$. 

We note that for all $j\in\mathbb N$, setting $t^k_j:=\e_k(jT+(j-1)K)$, by \eqref{312} we have
$$
v_{\e_k}( t^k_j) =\e_k w^1(0)+ \e_k\sum_{\ell=1}^j (w^\ell(T)-w^\ell(0)),
$$
so that,   since $w^\ell\in X_0^{\ell,T}(\xi)$,
\begin{eqnarray*}
\EE(v_{\e_k}( t^k_j))&=&\e_k\EE(w^1(0))+ \e_k\sum_{\ell=1}^j\EE (w^\ell(T)-w^\ell(0))
\\
&=&\e_k\EE(w(0))+ \e_k j \EE (w(T)-w(0))
\end{eqnarray*}
and, recalling that  $w\in X_0^T(\xi)$, we have 
$$
\EE(v_{\e_k}( t^k_j))=\e_k\EE(w^1(0))+ \e_kj T \xi=\e_k\EE(w^1(0)) +v_{\xi,T,K}(t^k_j) +{\e_k\frac{K}{T+K}}.
$$
  Since for all $t\in(a,b)$ and for all $k$ there exists $t^k_j$ such that $|t-t^k_j|\le \e_k(T+K)$, by the equi-H\"older continuity of $t\mapsto \EE(v_{\e_k}( t))$ {(see Proposition \ref{pro1})}  we have
$$
|\EE(v_{\e_k}( t))- \EE(v_{\e_k}( t^k_j))|\le \sqrt{M\e_k(T+K)(|\xi|^2+1)}
$$
We then obtain the claim as $k\to+\infty$. 
\end{proof}

\begin{proposition}\label{pro3}
With fixed $0\le a<b$, we have 
$$
\limsup_{\e\to 0}F_\e(v_\e,\mathcal F^\e_t,B^\e_t;{(a,b)})\le 
(b-a)\Big(f_{\rm hom}(\xi)+\eta+ \Big(2\frac{K}{T}+2\eta_K\Big)\max|V_{\rm per}|\Big).
$$
\end{proposition}

\begin{proof} 
Using a change of variables and the periodicity of $V_{\rm per}$, by \eqref{tre-b}  with $s_1=a/\e$ and $s_2=b/\e$ we obtain
\begin{eqnarray*}
\limsup_{\e\to0} F_\e(v_\e,\mathcal F^\e_t,B^\e_t; (a,b))
&=& \limsup_{\e\to0}\e\, \EE\Big(\int_{a/\e}^{b/\e} \Big(\frac12|\dot{w}{}^\xi(t)|^2+V_{\rm per}(B^\xi_t+ w^\xi(t))\Big)dt\Big)\\
&\le&{
(b-a)\Big( f_{\rm hom}(\xi)+\eta +\Big(3\frac{K}T+{2}\eta_K\Big)\max|V_{\rm per}|\Big)},
\end{eqnarray*}
and the claim
\end{proof}

\begin{corollary}[recovery sequence for an affine target function] For all $\xi$ and $x$ and for all $\delta>0$ there exists a sequence $\overline v_\e\to x+ \xi t$ and such that 
$$
\limsup_{\e\to 0}F_\e(\overline v_\e,\mathcal F^\e_t,B^\e_t; (a,b))\le  f_{\rm hom}(\xi)(b-a).
$$
\end{corollary}

\begin{proof} Let $\eta=\delta/4$, and let $K$ be large enough so that $2\eta_K\max|V_{\rm per}|\le \delta/4$.
Let $T_j$ be a diverging sequence such that 
\begin{eqnarray*}
&&f_{\rm hom}(\xi)=\lim_{j\to+\infty} \frac1{T_j}\inf\bigg\{\int_0^{T_j} \EE\Big(\frac12|\dot w(t)|^2+V_{\rm per}(B^j_t+ w(t))\Big)dt: w\in X_0^{T_j}(\xi),  B^j_t ,  \mathcal F_t \bigg\}.
\end{eqnarray*}

By Proposition \ref{pro2} for all fixed $j$, the sequence $v^j_{\e}$ defined therein with $T=T_j$ converges in the sense of Definition {\rm\ref{defco}} to the function given by $v_{\xi,T,K}$ given by 
$$
v^j_{\xi,K}(t)= x+ \frac{T_j}{T_j+K} t \xi .
$$
Moreover, letting $\mathcal F^{j,\e}_t$ and $B^{j,\e}$ be the corresponding scaled filtration and Brownian motion,
by Proposition \ref{pro3} we have 
$$
\limsup_{\e\to 0}F_\e(v^j_\e,\mathcal F^{j,\e}_t,B^{j,\e}_t; (a,b))\le \Big(f_{\rm hom}(\xi)+\frac\delta2+ 2\frac{K}{T_j}\max|V_{\rm per}|\Big)(b-a).
$$
The conclusion follows by letting $j\to+\infty$ and using a diagonal argument.
\end{proof}

\begin{proposition}\label{pro5}
With fixed $0\le a<b$, and $v$ piecewise affine, there exists $v_\e$ converging to $v$ in the sense of Definition {\rm\ref{defco}}, filtrations $\mathcal F^\e_t$, and Brownian motions $B^\e_t$ such that
$$
\limsup_{\e\to 0}F_\e(v_\e,\mathcal F^\e_t,B^\e_t; (a,b))\le 
\int_a^b f_{\rm hom}(v'(t))dt.
$$
\end{proposition}

\begin{proof} We define $a=a_0<a_1<\dots < a_m=b$ and $\xi_1,\ldots,\xi_m$ such that $v'(t)=\xi_i$ in $(a_{i-1}, a_i)$. It is not restrictive to suppose that $\tfrac{a_{i}}\e>\tfrac{a_{i-1}}\e+K$ for all $i\in\{1,\ldots,m\}$.
Note that, up to a translation, it is not restrictive to suppose that $v(t)=\xi_1t$ in the first interval. 

With fixed $\eta>0$, $T_0$ and $K$ we define $T_i>T_0$ and the correctors $\mathcal F^{\xi_i}_t$, $B^{\xi_i}_t$ and $w^{\xi_i}$ as in Proposition \ref{correct} for $i\in\{1,\ldots,m\}$. It is not restrictive to suppose that the filtrations  $\mathcal F^{\xi_i}_t$ are mutually independent. For all $j$ we define 
$$
\widehat{\mathcal F}^j_t=\sigma\Big(\bigcup_{i=1}^j\mathcal F^{\xi_i}_t\Big).
$$

With fixed $\e>0$ we define $w^\e$, $\mathcal F^{\e}_t$, and $B^{\e}_t$ by a finite induction argument.
With fixed $r>0$ let $\tau^\e_1$ be the stopping time with respect to $\widehat{\mathcal F}^2_t$ defined as the smallest $t\ge a_1/\e$ such that
$$
{\rm dist}(B^{\xi_1}_t+w^{\xi_1}(\tfrac{a_1}\e)-w^{\xi_2}(0),\mathbb Z^d)\le r
$$
and set $\widetilde\tau^\e_1=\tau^\e_1\wedge(\tfrac{a_1}\e+K)$.
For $t\ge 0$ we set 
$$
w^\e_1(t)=w^{\xi_1}(t) \hbox{ and }B^{\e,1}_t=B^{\xi_1}_t.
$$

For $t\ge0$ we then let 
$$
w^\e_2(t)=w^{\xi_2}((t-\widetilde\tau^\e_1)^+)-w^{\xi_2}(0)+w^{\e,1}(t\wedge\widetilde\tau^\e_1), \hbox{ and }B^{\e,2}_t=B^{\xi_2}_{(t-\widetilde\tau^\e_1)^+}+ B^{\e, 1}_{t\wedge\widetilde\tau^\e_1}.
$$
and let
$\tau^\e_2$ be the stopping time  with respect to $\widehat{\mathcal F}^3_t$ defined as the smallest $t\ge a_2/\e$ such that
$$
{\rm dist}(B^{\e,2}_t+w^\e_2(\tfrac{a_2}\e)-w^{\xi_3}(0),\mathbb Z^d)\le r
$$
and set $\widetilde\tau^\e_2=\tau^\e_2\wedge(\tfrac{a_2}\e+K)$

This construction can then be iterated, by setting
for $t\ge0$  
$$
w^\e_i(t)=w^{\xi_i}((t-\widetilde\tau^\e_{i-1})^+)-w^{\xi_i}(0)+w^{\e,{i-1}}(t\wedge\widetilde\tau^\e_{i-1})\hbox{ and }B^{\e,i}_t=B^{\xi_i}_{(t-\widetilde\tau^\e_{i-1})^+}+ B^{\e, {i-1}}_{t\wedge\widetilde\tau^\e_{i-1}}.
$$
and letting
$\tau^\e_i$ be the stopping time  with respect to $\widehat{\mathcal F}^{i+1}_t$ defined as the smallest $t\ge a_i/\e$ such that
$$
{\rm dist}(B^{\e,i}_t+w^\e_i(\tfrac{a_2}\e)-w^{\xi_{i+1}}(0),\mathbb Z^d)\le r
$$
and set $\widetilde\tau^\e_i=\tau^\e_i\wedge(\tfrac{a_i}\e+K)$.
We finally set 
$$
w^\e(t):=w^\e_m(t), \ \ B^{\e}_t=B^{\e,m}_t,\ \hbox{ and }\mathcal F^\e_t=\widehat{\mathcal F}^{i+1}_t.
$$

If we consider $v_\e(t)= \e w^\e(t/\e)$, then $v_\e\to v_{K,\{T_i\}}$, the piecewise affine function with
$$
v'_{K,\{T_i\}}=  \frac{T_i}{T_i+K}\xi_i 
$$
on $(a_{i-1}, a_i)$.  After a diagonalization, the same computation as in Proposition \ref{pro3} gives the conclusion.
\end{proof}

\begin{corollary}\label{coro5}
The function $f_{\rm hom}\colon \mathbb R^d\to\mathbb R$ is convex.
\end{corollary}

\begin{proof} With fixed $\xi_1$, $\xi_2\in  \mathbb R^d$, $\lambda\in(0,1)$, for all $k\in\mathbb N$ we consider the piecewise-affine function $u_k\colon (0,1)\to \mathbb R^d$, with $u_k(0)=0$ and whose derivative is the $\frac1k$-periodic function equal to $\xi_1$ on $(0,\lambda/k)$ and $\xi_2$ on $(\lambda/k,1/k)$. By Proposition \ref{pro5} for all $k\in\mathbb N$ there exists $u_{k,\e}$ converging to $u_k$ in the sense of Definition {\rm\ref{defco}}, filtrations $\mathcal F^{k,\e}_t$, and Brownian motions $B^{k,\e}_t$ such that
$$
\limsup_{\e\to 0}F_\e(u_{k,\e},\mathcal F^{k,\e}_t,B^{k,\e}_t; (0,1))\le \lambda f_{\rm hom}(\xi_1)+(1-\lambda) f_{\rm hom} (\xi_2).
$$
Note that the right-hand side is independent of $k\in\mathbb N$.
Since $u_k$ converge weakly in $H^1((0,1);\mathbb R^d)$ to to the function $u(t)=\xi t$, with $\xi=  \lambda  \xi_1+(1-\lambda) \xi_2$,  by a diagonal argument we can select a sequence $\e_k\to0$ such that the corresponding $u_{k,\e_k}$ converge to $u$ in the sense of Definition \ref{defco}, and 
$$
\limsup_{k\to+\infty}F_\e(u_{k,\e_k},\mathcal F^{k,\e_k}_t,B^{k,\e_k}_t; (0,1))\le \lambda f_{\rm hom}(\xi_1)+(1-\lambda) f_{\rm hom} (\xi_2).
$$
Since by Theorem \ref{main}(i) we have
$$
\liminf_{k\to+\infty}F_\e(u_{k,\e_k},\mathcal F^{k,\e_k}_t,B^{k,\e_k}_t; (0,1))\ge f_{\rm hom}(\xi)
$$
we conclude that $ f_{\rm hom}(\xi)\le \lambda f_{\rm hom}(\xi_1)+(1-\lambda) f_{\rm hom} (\xi_2)$; that is, the claim.
\end{proof}

\begin{proof}[Proof of Theorem \rm \ref{main}(ii)] By the convexity and the upper bound of $f_{\rm hom}$, proved in Corollary \ref{coro5} and \eqref{13},  the functional $u\mapsto \int_a^bf_{\rm hom}(u')dt$ is strongly continuous in $H^1((a,b);\mathbb R^d)$. Hence, the conclusion follows from Proposition \ref{pro5} and the strong density of piecewise-affine functions, using the same diagonal argument as in the proof of Corollary \ref{coro5}.
\end{proof}

\section{A stability result for Lagrangians with defects}\label{defects}

\def\FB{\widetilde F}

We consider a bounded Borel function $V_0\colon\mathbb R^d\to [0,+\infty)$ such that $V_0(x)\to 0$ as $|x|\to+\infty$, and the perturbed functional
\begin{equation}\label{deffe-dif}
\FB_\e(v,\mathcal F_t,B_t)= \int_a^{b} \EE\Big(\frac12 |\dot v(t)|^2+ V_{\rm per} \Big(\frac{\sqrt\e B_t+ v(t)}\e\Big)+V_0 \Big(\frac{\sqrt\e B_t+ v(t)}\e\Big)\Big)dt
\end{equation}

Then the $\Gamma$-limit of the functionals $\FB_\e$ is the same as that obtained for the periodic case in Theorem \ref{main}. More precisely, we have the following result.

\begin{theorem}\label{main-defect} Let $\FB_\e$ be given by \eqref{deffe-dif}. Given $\e_k\to 0$, we have

\smallskip
{\rm(i)} For every $v\in H^1(a,b)$ and for every sequences $\mathcal F^k_t$,  $B^k_t$, and $v_k\to v$ such that 
$v_k(t)$ is $\mathcal F^k_t$-measurable, we have 
$$
F_{\rm hom} (v)\le\liminf_{k\to+\infty}\FB_{\e_k}(v_k, \mathcal F^k_t,B^k_t);
$$

\smallskip
{\rm(ii)} for every $v\in H^1(a,b)$ there exists sequences  $\mathcal F^k_t$,  $B^k_t$, and $v_k\to v$, such that 
$v_k(t)$ is $\mathcal F^k_t$-measurable,  such that
$$
F_{\rm hom} (v)\ge\limsup_{k\to+\infty}
\FB_{\e_k}(v_k, \mathcal F^k_t,B^k_t).
$$
\end{theorem}
\begin{proof} Since $V_0\ge 0$ (i) follows from the analogous property in Theorem \ref{main}.

On order to prove (ii), we fix $\e_k$, $v_k$, $B^k_t$ and $\mathcal F^k_t$ with $v_k$ converging to some $v$. For every $s\ge 0$, let $\mu^k_s$ be the probability law of $ 
(\sqrt{\e_k}B^k_{t}+v_k(s))/\e_k$;
that is, for every Borel set $A$ in $\mathbb R$ 
$$
\mu^k_t(A):= P\Big(\Big\{\frac{\sqrt{\e_k}B^k_{t}+v_k(s)}{\e_k}\in A\Big\}\Big).
$$
We also introduce the probability measure defined by
$$
\mu^k(A):= \frac{1}{b-a}\int_{a}^{b}\mu^k_t(A)dt.
$$

We now fix {\color{red} $\eta>0$}, and $M\in\mathbb N$ such that $M^{-d}\le\eta$. Since $\mu^k([-ML,ML]^d)\le 1$, 
there exists $z_k\in\mathbb Z^d\cap[-ML,ML]^d$ such that
$$
\mu^k(z_k+[-L,L]^d)\le \frac1{M^d}\le\eta.
$$
We can then use as recovery sequence the same $B^k_t$ and $\mathcal F^k_t$, and the function $\widetilde v_k(t)= v_k(t)-\e_k z_k$, obtained by translating the function $v_k$. We note that $\widetilde v_k$ still converges to $v$ as $k\to +\infty$ since $z_k$ is bounded, and that 
$$
V_{\rm per} \Big(\frac{\sqrt{\e_k} B^k_t+ v_k(t)}{\e_k}-z_k\Big)=V_{\rm per} \Big(\frac{\sqrt{\e_k} B^k_t+ v_k(t)}\e\Big)
$$
by the periodicity of $V_{\rm per} $. Hence,
\begin{eqnarray*}
\FB_{\e_k}(\widetilde v_k,\mathcal F^k_t,B^k_t)=  F_{\e_k}(\widetilde v_k,\mathcal F^k_t,B^k_t)+\int_a^{b} \EE\Big(V_0 \Big(\frac{\sqrt{\e_k} B^k_t+ v_k(t)}{\e_k}-z_k\Big)\Big)dt
\end{eqnarray*}  
We now observe that
\begin{eqnarray*}
\int_a^{b} \EE\Big(V_0 \Big(\frac{\sqrt{\e_k} B^k_t+ v_k(t)}{\e_k}-z_k\Big)\Big)dt
\le (b-a)\eta \sup V_0.
\end{eqnarray*}  
Hence, we have
$$
F_{\e_k}(v_k, \mathcal F^k_t,B^k_t)+ (b-a)\eta \sup V_0
\ge \FB_{\e_k}(\widetilde v_k, \mathcal F^k_t,B^k_t),
$$
which proves (ii) by choosing as $v_k$ as in the corresponding claim in Theorem \ref{main} {by the arbitrariness of $\eta$}.
\end{proof}

\section{Convergence in $(0,+\infty)$}\label{stab-hl}
We now consider the extension of the $\Gamma$-convergence result to some functionals defined on functions on the half line, in view of the applications to steady-state Hamilton-Jacobi equations.
\begin{definition}\label{Hlambda}
Let  $\lambda>0$ be a given parameter. The space
$H^1_\lambda((0,+\infty);\mathbb R^d)$ 
is defined as the space of
of all functions $u\colon(0,+\infty)\to\mathbb R^d$
 such that  $u\in H^1((0,t_0);\mathbb R^d)$ for all $t_0>0$,
and 
$$
\int_0^{+\infty} |u'(t)|^2  e^{-\lambda t}dt<+\infty,$$
endowed with the norm
$$
\Big(|u(0)|^2 +\int_0^{+\infty} |u'(t)|^2  e^{-\lambda t}dt\Big)^{1/2}.
$$
Similarly, the space
$H^1_\lambda((0,+\infty);L^2(\Omega;\mathbb R^d))$ is defined as the space of
of all functions $u\colon(0,+\infty)\to L^2(\Omega;\mathbb R^d)$,
 such that $u\in H^1((0,t_0);L^2(\Omega;\mathbb R^d))$ for all $t_0>0$,
and 
$$
 \int_0^{+\infty} \EE(|u'(t)|^2)  e^{-\lambda t}dt<+\infty,
$$
endowed with the norm
$$
\Big(\EE (|u(0)|^2) +\int_0^{+\infty} \EE(|u'(t)|^2)  e^{-\lambda t}dt\Big)^{1/2}.
$$
We observe that $H^1_\lambda((0,+\infty);\mathbb R^d)$ and $H^1_\lambda((0,+\infty);L^2(\Omega;\mathbb R^d))$ are {Hilbert} spaces. 
\end{definition}

\begin{definition}[convergence to a deterministic variable on $(0,+\infty)$]\label{defcoinfty}
Let $\lambda>0$ be fixed. Given a sequence $v_k\in H^1_\lambda((0,+\infty);L^2(\Omega;\mathbb R^d))$ and $v\in H^1_\lambda((0,+\infty);\mathbb R^d)$ we write $v_k\to v$ if 
for all subsequences $v_{k_j}$, there exist a subsequence, that we still denote by $v_{k_j}$, 
and $\widehat v\in H^1_\lambda((0,+\infty);L^2(\Omega;\mathbb R^d))$, possibly depending on the subsequence, such that 
$$
 v_{k_j}\wto \widehat v \hbox{ weakly in }H^1_\lambda((0,+\infty);L^2(\Omega;\mathbb R^d))
\quad\hbox{ and }\quad
\EE(\widehat v(t))=v(t)\hbox{ for every }t.
$$
\end{definition}

\subsection{Homogenization on the half line}
Given $\lambda>0$, for all $\e>0$, for all $v\in H^1_\lambda((0,+\infty);L^2(\Omega;\mathbb R^d))$, $\{\mathcal F_t\}_{t\ge 0}$ filtrations such that $v(t)$ is measurable with respect to $\mathcal F_t$ for every $t$, and  $\{B_t\}_{t\ge 0}$ Brownian motions adapted to $\{\mathcal F_t\}_{t\ge 0}$,
we set
\begin{equation}\label{Fle}
F^\lambda_\e(v,\mathcal F_t,B_t)= \int_0^{+\infty} \EE\Big(\frac12 |\dot v(t)|^2+ V_{\rm per} \Big(\frac{\sqrt\e B_t+ v(t)}\e\Big)\Big)\, e^{-\lambda t}dt.
\end{equation}
For $v\in H^1_\lambda((0,+\infty);\mathbb R^d)$, we also define
\begin{equation}\label{Flom}
F_{\rm hom}^\lambda(v)=\int_0^{+\infty} f_{\rm hom}(\dot v(t))\,e^{-\lambda t}dt.
\end{equation}
We will prove the following homogenization result, whose stability with respect to perturbations can be proved using the assumptions and the results in Section \ref{defects}.

\begin{theorem}\label{main2} Given $\lambda>0$, let $F^\lambda_\e$ and $F_{\rm hom}^\lambda$ be given by \eqref{Fle} and \eqref{Flom}, respectively. Given $\e_k\to 0$, for every $v\in H^1_\lambda((0,+\infty);\mathbb R^d)$ we have

\smallskip
{\rm(i)}  for every sequence of functions $v_k\in H^1_\lambda((0,+\infty);L^2(\Omega;\mathbb R^d))$ with  $v_k\to v$ in the sense of Definition {\rm\ref{defcoinfty}},   for every sequence $\{\mathcal F^k_t\}_{t\ge 0}$ of filtrations such that each $v_k(t)$ is $\mathcal F^k_t$-measurable, and  for every sequence $\{B^k_t\}_{t\ge 0}$ of Brownian motions adapted to $\{\mathcal F^k_t\}_{t\ge 0}$,
we have 
$$
F^\lambda_{\rm hom} (v)\le\liminf_{k\to+\infty}F^\lambda_{\e_k}(v_k, \mathcal F^k_t,B^k_t);
$$

\smallskip
{\rm(ii)} there exist a sequence of functions $v_k\in H^1_\lambda((0,+\infty);L^2(\Omega;\mathbb R^d))$ with  $v_k\to v$ in the sense of Definition {\rm\ref{defcoinfty}},  a sequence $\{\mathcal F^k_t\}_{t\ge 0}$ of filtrations such that each $v_k(t)$ is $\mathcal F^k_t$-measurable, and a sequence $\{B^k_t\}_{t\ge 0}$ of Brownian motions adapted to $\{\mathcal F^k_t\}_{t\ge 0}$, such that 
$$
F^\lambda_{\rm hom} (v)\ge\limsup_{k\to+\infty}F^\lambda_{\e_k}(v_k, \mathcal F^k_t,B^k_t).
$$
\end{theorem}

\begin{proof}
Claim (i). Let $\{a_j\}_{j\ge0}$ be an increasing sequence in $[0,+\infty)$ with $a_0=0$ and $a_j\to+\infty$ as $j\to+\infty$. Using the notation
\begin{equation}\label{deffeab}
F_\e(v,\mathcal F_t,B_t; (a,b))= \int_a^{b} \EE\Big(\frac12 |\dot v(t)|^2+ V_{\rm per} \Big(\frac{\sqrt\e B_t+ v(t)}\e\Big)\Big)dt
\end{equation}
and Theorem \ref{main}, we have
\begin{eqnarray*}
\liminf_{k\to+\infty}F^\lambda_{\e_k}(v_k, \mathcal F^k_t,B^k_t)
\ge \sum_{j\ge 1}\liminf_{k\to+\infty}F_{\e_k}(v_k, \mathcal F^k_t,B^k_t;(a_{j-1},a_j)) e^{-\lambda a_j}\\
\ge \sum_{j\ge 1}\int_{a_{j-1}}^{a_j}f_{\rm hom} (\dot v(t)) e^{-\lambda a_j}dt
\
\ge \sum_{j\ge 1}e^{-\lambda(a_j-a_{j-1})}\int_{a_{j-1}}^{a_j}f_{\rm hom} (\dot v(t)) e^{-\lambda t}dt\\
\ge \sup_{j\ge1}e^{-\lambda(a_j-a_{j-1})}\int_{0}^{+\infty}f_{\rm hom} (\dot v(t)) e^{-\lambda t}dt.
\end{eqnarray*}
Since $\sup_{j\ge1}e^{-\lambda(a_j-a_{j-1})}$ can be taken arbitrarily close to $1$, we have proved the claim.

\medskip
Claim (ii).  By a density argument (see Proposition 4.5 in \cite{MR5057984}), it is sufficient to consider  a target function $v\in H^1_\lambda((0,+\infty);\mathbb R^d)$ such that there exists $t_0$ such that $v(t)=0$ if $t\ge t_0$. We fix $t_1>t_0$. By Theorem \ref{main} with $(a,b)=(0,t_1)$, there exist $v_k$, $\{\mathcal F^k_t\}$, $\{B^k_t\}$ as in Claim (ii) of that theorem. By Lemma \ref{lemmada} we can assume that $\EE(v_k(0))=v(0)$ and $\EE(v_k(t_1))=v(t_1)=0$. The function $v_k$ is then extended by setting $v_k(t)=v_k(t_1)$ for $t\ge t_1$, and, recalling Definition \ref{Hlambda}, we have that $v_k\in H^1_\lambda((0,+\infty);L^2(\Omega))$, and $v_k\to v$ in the sense of  Definition {\rm\ref{defcoinfty}}. Let $N\in\mathbb N$, and set $a_j=jt_1/N$. We note that, using the validity of Claim (i) of Theorem \ref{main} for $(a,b)=(a_{j-1}, a_j)$,  and the validity of Claim (ii) of the same theorem with $(a,b)=(0, t_1)$ we can prove $v_k$ satisfies Claim (ii) of Theorem \ref{main} with $(a,b)=(a_{j-1}, a_j)$ for any $j\in\{1,\ldots, N\}$. We then have
\begin{eqnarray*}
&&\limsup_{k\to+\infty}F^\lambda_{\e_k}(v_k, \mathcal F^k_t,B^k_t)
\\
&\le& \sum_{j= 1}^N\limsup_{k\to+\infty}F_{\e_k}(v_k, \mathcal F^k_t,B^k_t;(a_{j-1},a_j)) e^{-\lambda a_{j-1}}
+\int_{t_1}^{+\infty} e^{-\lambda t}\max|V_{\rm per}|dt
\\
&\le& \sum_{j=1}^N\int_{a_{j-1}}^{a_j}f_{\rm hom} (\dot v(t)) e^{-\lambda a_{j-1}}dt +e^{-\lambda t_1}\frac{\max|V_{\rm per}|}\lambda
\\
&\le& e^{t_1/N} \int_{0}^{t_1}f_{\rm hom} (\dot v(t)) e^{-\lambda t}dt+e^{-\lambda t_1}\frac{\max|V_{\rm per}|}\lambda.
\end{eqnarray*}
By the arbitrariness of $N$ and $t_1>t_0$ we conclude the proof.
\end{proof}

\subsection{Convergence of value functions}
Given $\lambda>0$,  define
\begin{equation}\label{Uex}
U_\e(x)=\inf\big\{ F^\lambda_\e(v,\mathcal F_t,B_t): \EE(v(0))=x\big\},
\end{equation}
where the infimum is taken over all $v\in H^1_\lambda((0,+\infty);L^2(\Omega;\mathbb R^d))$, satisfying the initial condition, $\{\mathcal F_t\}_{t\ge 0}$ filtrations such that $v(t)$ is measurable with respect to $\mathcal F_t$ for every $t$, and  $\{B_t\}_{t\ge 0}$ Brownian motions adapted to $\{\mathcal F_t\}_{t\ge 0}$.
Similarly, we define
\begin{equation}
U_{\rm hom}(x)=\inf\big\{ F^\lambda_{\rm hom}(v): v\in H^1_\lambda((0,+\infty);\mathbb R^d), v(0)=x\big\},
\end{equation}

\begin{theorem} Given $\lambda$ and $x$ we have $U_\e(x)\to U_{\rm hom}(x)$ as $\e\to0$.
\end{theorem}

\begin{proof} With fixed $\e_k$ converging to $0$, we choose all $v_k\in H^1_\lambda((0,+\infty);L^2(\Omega;\mathbb R^d))$, with $\EE(v_k(0))=x$, $\{\mathcal F^k_t\}_{t\ge 0}$ filtrations such that $v(t)$ is measurable with respect to $\mathcal F^k_t$ for every $t$, and  $\{B^k_t\}_{t\ge 0}$ Brownian motions adapted to $\{\mathcal F^k_t\}_{t\ge 0}$ such that 
$$
F^\lambda_{\e_k}(v_k,\mathcal F^k_t,B^k_t)\le U_{\e_k}(x)+\frac1k .
$$
Using the constant function $x$ as competitor in \eqref{Uex}, from \eqref{Fle} we obtain that $U_{\e_k}(x)$ is bounded uniformly. Again, using \eqref{Fle} we then obtain
$$
\int_0^{+\infty} \EE\Big(\frac12 |\dot v_k|^2+ V_{\rm per} \Big(\frac{\sqrt\e B^k_t+ v_k}\e\Big)\Big)\, e^{-\lambda t}dt\le C,
$$
from which we deduce that $v_k$ is bounded in $H^1_\lambda((0,+\infty);L^2(\Omega;\mathbb R^d))$, and hence we can suppose that it converges to a function $\widehat v$ weakly in  $H^1_\lambda((0,+\infty);L^2(\Omega;\mathbb R^d))$. This implies that $v_k\to v$ in the sense of  Definition {\rm\ref{defcoinfty}}, where $v(t):= \EE(\widehat v(t))$ for every $t\ge 0$, and $v(0)=\EE(\widehat v(0))=x$. From Theorem \ref{main2}(i) we have that 
\begin{equation}\label{lifi}
U_{\rm hom}(x)\le  F^\lambda_{\rm hom}(v)\le \liminf_{k\to+\infty} F^\lambda_{\e_k}(v_k,\mathcal F^k_t,B^k_t)\le \liminf_{k\to+\infty} U_{\e_k}(x).
\end{equation}

To prove 
\begin{equation}\label{lisu}
U_{\rm hom}(x)\ge \limsup_{k\to+\infty} U_{\e_k}(x),
\end{equation}
with fixed $\eta>0$ we consider  $v\in H^1_\lambda((0,+\infty);\mathbb R^d)$, with $v(0)=x$, such that
$$
U_{\rm hom}(x)+\eta \ge F^\lambda_{\rm hom}(v).
$$
Let $v_k$, $\{\mathcal F^k_t\}_{t\ge 0}$, and $\{B^k_t\}_{t\ge 0}$ be as in Theorem \ref{main2}(ii).
By Lemma \ref{lemmada} we can suppose that $\EE(v_k(0))=x$ for every $k$. Thus we have
$$
U_{\rm hom}(x)+\eta\ge F^\lambda_{\rm hom}(v) \ge\limsup_{k\to+\infty}F^\lambda_{\e_k}(v_k, \mathcal F^k_t,B^k_t)\ge \limsup_{k\to+\infty}U_{\e_k}(x).
$$
By the arbitrariness of $\eta>0$ we get \eqref{lisu}, which, together with \eqref{lifi}, proves the claim.
\end{proof}

\begin{remark}[addition of continuous perturbations]\rm
For  $v\in H^1_\lambda((0,+\infty);L^2(\Omega;\mathbb R^d))$ and $g\in H^1_\lambda((0,+\infty);\mathbb R^d)$, we set
\begin{equation}\label{Fle2}
G^\lambda(v)= \int_0^{+\infty} \EE(
\langle g(t),\dot v(t)\rangle)\, e^{-\lambda t}dt.
\end{equation}
We note that if $v_k\to v$ then
$$
\lim_{k\to+\infty}G^\lambda(v)= \int_0^{+\infty} 
\langle g(t),\dot v(t)\rangle\, e^{-\lambda t}dt.
$$

Indeed, if $v_k\wto \widehat v$ in $H^1_\lambda((0,+\infty);L^2(\Omega;\mathbb R^d))$ then $v_k\wto \widehat v$ in $L^2_\lambda((0,+\infty);L^2(\Omega;\mathbb R^d))$, so that  $G^\lambda(v_k)\to G^\lambda(\widehat v)$, from which we conclude since
$$
G^\lambda(\widehat v)=\int_0^{+\infty} 
\langle g(t),\dot v(t)\rangle\, e^{-\lambda t}dt.
$$

Given $\lambda>0$,  define
\begin{equation}\label{73}
U_\e(x)=\inf_{\mathcal F_t,  B_t}\big\{ F^\lambda_\e(v,\mathcal F_t,B_t) + G^\lambda (v):  \EE(v(0))=x\big\},
\end{equation}
where the infimum is taken over all $v\in H^1_\lambda((0,+\infty);L^2(\Omega;\mathbb R^d))$,  such that $v(t)$ is measurable with respect to $\mathcal F_t$ for every $t$ and satisfies the initial condition, as well as over all  filtrations $\{\mathcal F_t\}_{t\ge 0}$ of $\mathcal F$ and  over all  Brownian motions $\{B_t\}_{t\ge 0}$ adapted to $\{\mathcal F_t\}_{t\ge 0}$.
Similarly, we define
\begin{equation}\label{74}
U_{\rm hom}(x)=\inf\big\{ F^\lambda_{\rm hom}(v) +G^
\lambda(v): v\in H^1_\lambda((0,+\infty);\mathbb R^d), v(0)=x\big\}.
\end{equation}
We then have $U_\e(x)\to U_{\rm hom} (x)$. 
\end{remark}

\subsection{Characterization of the limit Hamiltonian and stability}
We take $g(t)$ equal to a constant $-p\in\mathbb R^d$ in the previous remark. In this case, $U_\e=U_{p,\e}$ is the solution to a problem of the form 
\eqref{Uex} with $\frac12|\dot v|^2-\langle p, \dot v\rangle$ in the place of $\frac12|\dot v|^2$ in the energy.   Correspondingly, $U_{\hom}(x)$ is the solution $U_p(x)$ of the limit problem with $f_{
\hom}(\dot v)-\langle p, \dot v\rangle$ in the place of $f_{\hom}(\dot v)$. Note that $U_p$ is the constant
$$
U_p=\frac1\lambda\min\{ f_{\rm hom}(\xi)-\langle p, \xi\rangle\}= -\frac1\lambda f_{\rm hom}^*(p) .
$$
Since, on the other hand, $U_p$ is also characterized as solving the limit Hamilton-Jacobi equation \eqref{Homp}, we deduce that  $H_{\rm hom}=f_{\rm hom}^*$.

As a consequence of the convergence result for Hamilton--Jacobi equations in terms of viscosity solutions \cite{MR1159184}, we deduce that, for any fixed uniformly continuous and bounded $\psi$ we also have that the value functions
$$
U_{\e}(x)=\inf\Big\{
\int_0^{+\infty} \EE\Big(\frac12 |\dot v(t)|^2+ V_{\rm per} \Big(\frac{\sqrt\e B_t+ v(t)}\e\Big)+\psi(\sqrt\e B_t+ v(t))\Big)\, e^{-\lambda t}dt:  v(0)=x\Big\}
$$
converge to the value function
\begin{equation}\label{Uom}
U_{\rm hom}(x)=\inf\Big\{
\int_0^{+\infty} (f_{\rm hom}(\dot v(t))+\psi(v(t))\Big)\, e^{-\lambda t}dt:  v(0)=x\Big\}.
\end{equation}

\begin{remark}[stability for defects]
\rm From the arguments in Section \ref{defects}, we can conclude that, if $V_0$ is as in that section
and for any fixed uniformly continuous and bounded $\psi$, we also have that the value functions
\begin{eqnarray*}
&&U^0_{\e}(x)=\inf\Big\{
\int_0^{+\infty} \EE\Big(\frac12 |\dot v(t)|^2+ V_{\rm per} \Big(\frac{\sqrt\e B_t+ v(t)}\e\Big)\\
&&
\hskip3cm+ V_0\Big(\frac{\sqrt\e B_t+ v(t)}\e\Big)+\psi(\sqrt\e B_t+ v(t))\Big)\, e^{-\lambda t}dt:  v(0)=x\Big\}
\end{eqnarray*}
converge to the value function $U_{\rm hom}$ for the unperturbed case in \eqref{Uom}.

Indeed, for a fixed sequence $\e_k$, we can apply the arguments in the proof of Theorem \ref{main-defect} to a sequence of almost minimizers $v_k$, $B^k$ for $U^0_{\e_k}$, upon defining $\mu^k$ as
$$
\mu^k(A):= \frac{1}{\lambda}\int_{0}^{+\infty}\mu^k_t(A)e^{-\lambda t}dt.
$$

This gives a variational proof of a convergence result for $U^0_\e$ by P.L.~Lions \cite{PLL-college}.
\end{remark}

\section{Related settings}
We briefly describe how our proof could cover some more general Hamilton--Jacobi equations, and outline some possible extensions. Although definite conclusions are yet to be obtained on the various settings we mention here, we believe that bringing up possible connections of our arguments and results with these other {research lines might shed further light on} the mechanisms at play in our study and {we hope it} may motivate further research.
 
 \smallskip
We first consider the case of non-constant viscosity.
Throughout our work, $\sqrt\e B_t+v(t)$ stands for the solution to the stochastic differential equation
\begin{equation}
\label{eq:dXt-bis}
dX^\varepsilon_t\,=\,\alpha(t)\,dt\,+\,\sqrt\e\,dB_t,
\end{equation}
with $\dot v=\alpha$, starting at $X^\varepsilon_0=x$ at initial time~$t=0$. Equation \eqref{eq:dXt-bis} is the random trajectory to be considered in the stochastic-control interpretation of the Hamilton-Jacobi equation~\eqref{11}. 
We believe that our analysis carries over, \emph{mutatis mutandis}, to the case of the random trajectory solution to
 \begin{equation}
  \label{eq:dXt-sigma-not-1-bis}
dX^\varepsilon_t\,=\,\alpha(t)\,dt\,+\,\sqrt{\varepsilon}\,\sigma_{\rm per}\left(\frac{X^\varepsilon_t}{\varepsilon}\right)\,dB_t,
\end{equation}
for a scalar-valued periodic function~$\sigma_{\rm per}$, which we assume to be bounded, bounded away from zero, and sufficiently regular. The case of a suitable, matrix-valued~$\sigma_{\rm per}$ could likewise be considered, but we leave aside this further extension. Equation~\eqref{eq:dXt-sigma-not-1-bis} is an immediate generalization to~\eqref{eq:dXt-bis}, but also a non-trivial extension since, in sharp contrast to~\eqref{eq:dXt-bis}, the equation is implicit in~$X^\varepsilon_t$ and its solution cannot be simply expressed in terms of~$\alpha$ and $B_t$ as in the case \eqref{eq:dXt-bis}, where $X^\varepsilon_t=\int_0^t\alpha(s)ds+\sqrt\e B_t$. Equation~\eqref{eq:dXt-sigma-not-1-bis}  arises in the stochastic-control interpretation of the Hamilton-Jacobi equation
\begin{equation}
\label{eq:HJB-sigma-not-1-bis}
\lambda U_\varepsilon(x)\,-\,\frac12\,\varepsilon \,\sigma^2_{\rm per}\left(\frac{x}{\varepsilon}\right)\,\Delta U_\varepsilon(x) +\,H_{\rm per}\left(\frac{x}{\varepsilon},\nabla U_\varepsilon(x)\right)\,=\,\psi(x).
\end{equation}
 Our proof for~\eqref{eq:dXt-bis} may be adjusted to address the more general case~\eqref{eq:dXt-sigma-not-1-bis}, and to establish that the homogenized limit of~\eqref{eq:HJB-sigma-not-1-bis} takes the form~\eqref{12}, but now with the homogenized Hamiltonian~$H_{\rm hom}$ that is the Legendre transform of 
\begin{eqnarray}\label{fomd-sigma-not-1-bis}
f_{\rm hom}(\xi)=\liminf_{T\to+\infty} \frac1T \inf\bigg\{\int_0^T {\mathbb E}\Big(\frac12|\beta(t)|^2+V_{\rm per}(Y_t)\Big)dt \bigg\},
\end{eqnarray}
instead of~\eqref{fomma}. The infimum in~\eqref{fomd-sigma-not-1-bis} is taken, as for~\eqref{fomma} on the Brownian motion~$B_t$ and the filtration~$\mathcal F_t$, with all the necessary properties of adaptation described for~\eqref{fomma}, and (this is the new feature) over all the pairs~$(\beta(t),Y_t)$ that satisfy
\begin{equation}
dY_t\,=\,\beta(t)\,dt\,+\,\sigma_{\rm per}(Y_t)\,dB_t\quad\textrm{and}\quad{\mathbb E}\left(\frac{1}{T}\,\int_0^{T} \beta(s)\,ds\right)\,=\,\xi,
\end{equation}
where, of course, $\beta$ plays the role of $\dot w$ in our earlier setting.
The step of the proof that has essentially to be adjusted is the construction of the (almost) corrector.  
The concatenation of the trajectories at the stopping times is not performed in terms of the pair~$(w(t),B_t) $ but of the pair~$(\beta(t),Y_t)$. Notice indeed that our earlier construction, can itself be entirely expressed as  a concatenation of~$(\dot w(t),w(t)+B_t) $ and thus can be extended to our new setting.
Correspondingly, Proposition~\ref{apbrown} should be extended, this time using the properties of the partial differential equation~$-\frac{1}{2} \sigma_{\rm per}^2(x) \Delta u = 1 $, fully similar, in this respect, to those of~$-\frac{1}{2} \Delta u = 1 $, given our assumptions of non degeneracy of~$\sigma_{\rm per}$.

\smallskip
Another class of extensions we might consider for our work consists of general Hamiltonians~$H_{\rm per}(x,p)$ that are not quadratic in~$p$, or that do not even separately depend on the two variables~$p$ and~$x$, as was the case when~$H_{\rm per}(x,p)= \frac12|p|^2-V_{\rm per}(x)$. The latter Hamiltonian~$H_{\rm per}(x,p)$ being convex in~$p$, the oscillating Hamilton-Jacobi equation~\eqref{11} we consider derives from a stochastic control theory problem, and we use this property intensively in the present work. The running cost involved in this control problem and the drift within the random trajectories respectively write as~$L_{\rm per}(x,\alpha)=\frac12|\alpha|^2+V_{\rm per}(x)$ and~$b_{\rm per}(x,\alpha)=\alpha$. The general link between the Hamiltonian~$H_{\rm per}(x,p)$, the Lagrangian $L_{\rm per}(x,\alpha)$, and the drift $b_{\rm per}(x,\alpha)$ is given by the Legendre transformation
$$  H_{\rm per}(x,p)=\,\sup_{\alpha}\;\big\{\langle b_{\rm per}(x,\alpha),p\rangle-L_{\rm per}(x,\alpha)\big\}.$$
 Since our arguments proceed from $\Gamma$-convergence type arguments and are thus performed on the control-theoretic formulation of the problem, it is intuitive that other suitable running costs and drifts may be considered. Provided they enjoy the right analytic properties, our arguments will apply. Definite conclusions are yet to be obtained but we believe that, for instance, our arguments and results carry over to the case when~$b_{\rm per}(x,\alpha)$ is kept to be~$\alpha$ but~$L_{\rm per}(x,\alpha)$ is now 

$\bullet$  periodic and continuous in~$x$, with a modulus of continuity uniform in~$\alpha$, 

$\bullet$   there exist $c_1,c_2>0$ and $r>1$ such that $c_1|\alpha|^r\le L(x,\alpha)\le c_2(1+|\alpha|^r)$, and

$\bullet$   $L(x,\cdot)$ is convex for all $x$.

Considering more general drifts~$b_{\rm per}(x,\alpha)$ that may depend upon~$x$ (say, Lipschitz continuous in~$x$, uniformly in~$\alpha$) is likely to be more delicate. Likewise is the case when these extensions are combined with a diffusion~$\sigma_{\rm per}$ that is not taken constant any longer, as in the beginning of this section. {On the contrary, the stability result of Section \ref{defects} should extend unchanged to the setting of this section, since its proof only uses periodicity of $V_{\rm per}$ and non-negativity of $V_0$.}

\section{Appendix: properties of the Brownian motion}
As mentioned in the
Introduction, this appendix presents two classical results. The first result, namely
the following `concatenation property' is a standard result in the theory of stochastic processes. 
We provide a proof here for the convenience of the reader. 

  \begin{proposition}\label{concat}
 Consider $B^1_t$, $B^2_t$ two Brownian motions, valued in~${\mathbb R}^d$, adapted to two independent filtrations $\mathcal F^1_t$ and  $\mathcal F^2_t$,  and let $\tau$ be a stopping time for the filtration $\mathcal F^1_t$. Then the process
\begin{equation}
\widetilde{B}_t\,=\,\left\{
\begin{array}{cc}
B^1_t&\text{when}\quad t\leq\tau,  \\
 B^2_{t-\tau}\,+\,B^1_{\tau} &  \text{when}\quad t\geq\tau,
\end{array}
\right.
\end{equation}
is a Brownian motion, adapted to the filtration $\widetilde{\mathcal F}_t:=\sigma(\mathcal F^1_t,\mathcal F^2_t)$.
\end{proposition}

\begin{proof}
We are going to use a \emph{coupling method}. For this purpose, we first  introduce the concatenation mapping 
$F\colon \big(C\big({\mathbb R}_+,{\mathbb R}^d\big)\times{\mathbb R}_+\big)\times C_0\big({\mathbb R}_+,{\mathbb R}^d\big)\to C\big({\mathbb R}_+,{\mathbb R}^d\big)$,
where $C_0$ denotes continuous functions with value $0$ in $0$, defined by 
\begin{equation}
\label{eq:F}
F ((a,T),b) (t)\,
=\,\left\{
\begin{array}{cc}
a(t) &  \text{when} \quad t\,\leq \,T,  \\
b(t-T)\,+\,a(T) &  \text{when} \quad t\,\geq\,T.
\end{array}
\right.
\end{equation}
This function is deterministic, measurable and actually continuous. 
Note that $\widetilde{B}_t=F((B^1_t,\tau),B^2_t)$; that is, for $P$-almost every $\omega\in\Omega$ we have $\widetilde{B}_t(\omega)=F((B^1(\omega),\tau(\omega)) ,B^2(\omega))(t)$, where 
$B^i(\omega)(t)=B^i_t(\omega)$. From this, we can deduce that $\widetilde B_t$ is 
adapted to $\widetilde{\mathcal F}_t$. Note moreover that $\tau$ is also a stopping time for  $\widetilde{\mathcal F}_t$
\medskip

Let  $\{U_t\}_{t\ge 0}$ be the stochastic process with values in $(\mathbb R^d\times{\mathbb R}_+)\times \mathbb R^d$ defined by the pair
%We then note that the process $U_t$ defined by the pair
\begin{equation}
\label{eq:Ut}
U_t\,=\,\big((\widetilde B_{t \wedge\tau},\tau), \widetilde B_{t+\tau}-\widetilde B_{\tau}\big),
\end{equation}
and let $\{V_t\}_{t\ge 0}$ be the stochastic process with values in $(\mathbb R^d\times{\mathbb R}_+)\times \mathbb R^d$ defined by the pair
\begin{equation}
\label{eq:Vt}
V_t\,=\,\big((B^1_{t \wedge\tau},\tau), B^1_{t+\tau}- B^1_{\tau}\big).
\end{equation}
By well-known properties of the Brownian motions, the trajectories of $U_t$ and $V_t$  are almost surely elements of the functional space~$C\big({\mathbb R}_+,{\mathbb R}^d\times{\mathbb R}^d\big)$. Therefore, the stochastic processes $U_t$ and $V_t$ define two random functions with values in $\big(C\big({\mathbb R}_+,{\mathbb R}^d\big)\times{\mathbb R}_+\big)\times C\big({\mathbb R}_+,{\mathbb R}^d\big)$ denoted by $U$ and $V$.

Our claim $U_t\buildrel{\mathcal L}\over\sim V_t$ follows from the following three properties:

\smallskip
(i) $U_t$ also writes as
\begin{equation}
\label{eq:Ut-bis}
U_t\,=\big((B^1_{t \wedge\tau},\tau), B^2_t\big),\end{equation}
 respectively by definition of the first argument and because the second argument writes as
$\widetilde B_{t+\tau}-\widetilde B_{\tau} \,=\,(B^2_{t}\,+\,B^1_{\tau})-\,B^1_{\tau}\,=\,B^2_t$;

\smallskip(ii) by the strong Markov property, the second argument of $V_t$ is a Brownian motion~$W_t$, independent from $(B^1_{t\wedge\tau},\tau)$, so that
\begin{equation}
\label{eq:Vt-bis}
V_t\,=\, ((B^1_{t\wedge\tau},\tau), W_t);
\end{equation}

\smallskip (iii) the two arguments within each pair are independent from one another, so the joint law of each pair is the product of the laws of the arguments. Moreover, for every finite family $0\le t_1\le\cdots\le t_M$ by a similar argument we can see that the joint law in $\big((\mathbb R^d\times{\mathbb R}_+)\times \mathbb R^d\big)^M$ of 
$(U_{t_1},\ldots, U_{t_M})$ is the same as  the joint law  of $(V_{t_1},\ldots, V_{t_M})$.

\smallskip

We then deduce that the random maps $U$ and $V$ have the same law in $\big(C\big({\mathbb R}_+,{\mathbb R}^d\big)\times{\mathbb R}_+\big)\times C\big({\mathbb R}_+,{\mathbb R}^d\big)$.

\medskip

We finally remark that, since the concatenation mapping~$F$ is deterministic and continuous, the maps $F(U)$ and $F(V)$ have the same law in $C\big({\mathbb R}_+,{\mathbb R}^d\big)$. Since 
$$ \widetilde B \,=\,F(U)\quad \text{and}\quad B^1\,=\,F(V),$$
where $ \widetilde B $ and $B^1$ denote the continuous functions $t\mapsto  \widetilde B_t$ and $t\mapsto  B^1_t$, respectively, we conclude that $ \widetilde B$ and $ B^1$ have the same law in $C\big({\mathbb R}_+,{\mathbb R}^d\big)$. This implies that $\widetilde B$ is a Brownian motion.\end{proof}

The second property of the Brownian motion allows one to show that solutions to minimum problems can be extended until they `almost' reach a point in $\mathbb Z^d$, up to extending their domain in a controlled way, and up to a set of small probability. This shows that such solutions have a sort of almost-periodic behaviour, which can be used for the construction of correctors, as in the deterministic almost-periodic case.

\begin{proposition}\label{apbrown}
Given a Brownian motion $B_t$ on $\mathbb R^d$, a radius $r$, and a time length $T$, then there exists 
 $c(r,T)>0$, with $c(r,T)\to 0$ if $T\to+\infty$, such that
 $$
 P(\{\omega\in\Omega: {\rm dist}(B_t(\omega),\mathbb Z^d)\le r \hbox{ \rm for all } t\in [t_0,t_0+T]\}) \le c(r,T)
$$
 for each time $t_0$.
\end{proposition}

\begin{proof}
The issue is related to the capacity of the ball on the periodic torus. 
We let $W_t$, with $W_0=0$, denote the projection of the Brownian motion $B_t$ onto the torus $\mathbb{T}^d =
\mathbb{R}^d / \mathbb{Z}^d$. 

We first fix a closed ball $E_r\subset \mathbb{T}^d$ of radius $r>0$ and, for $x\in E_r^c:=\mathbb{T}^d\setminus E$, we let
\begin{equation}
\label{eq:Appendix-tau}
\tau_x := \inf \{ t \geq 0 : x\,+\,W_t \in E_r \}
\end{equation}
denote the first time at which the Brownian enters $E_r$, which is a stopping time for the filtration generated by the Brownian. A simple application of the Feynman-Kac representation formula for a steady-state elliptic equation, with homogeneous Dirichlet boundary conditions on~$\partial E_r$ and periodic conditions on the torus, implies that the function~$u(x)=\mathbb{E}[\tau_x]$ solves on~$\mathbb{T}^d$ the equation
$$
\begin{cases}
-\frac{1}{2} \Delta u = 1 & x\in E_r^c, \\
u = 0 & x \in E_r.
\end{cases}
$$
We then have

(a) the solution~$u$ is bounded and we may define 
$U_r := \sup\limits_{x \in \mathbb{T}^d} u(x) < \infty$;

(b) this bound only depends on the ambient dimension~$d$ and the radius~$r$ of the ball~$E$. 

\smallskip
In the following we use the shorthand $\{\tau_x > T_0\}$ for $\{\omega\in\Omega: \tau_x(\omega) > T_0\}$, and similar notation, for subsets of $\Omega$.

\smallskip
For any time $T_0$, we have 
$u(x)\,=\,\mathbb{E}[\tau_x]\,\geq\,T_0\,
{P}(\{\tau_x > T_0\})$,
so that, upon taking the supremum over~$x \in \mathbb{T}^d$, and noting that 
$
P(\{\tau_x > T_0\}) = {P}(
\{ x\,+\,W_t \in E_r^c\ \hbox{ for all } t \in [0,T_0]\})$,
we obtain
\begin{equation}
\label{eq:Appendix-sup}
 \sup_{x \in \mathbb{T}^d}P(
 \{ x\,+\,W_t \in E_r^c\ \hbox{ for all } t \in [0,T_0]\})\leq \frac{U_r}{T_0}.
\end{equation}
Using the Markov property of the Brownian motion, we now iterate formula~\eqref{eq:Appendix-sup}. To this end, we indeed observe that the probability of not entering the ball~$E_r$ in, say, an interval of length
$2\,T_0$ is smaller than that of not entering it in the first
interval $[0,T_0]$, and then of still not entering~$E_r$ in $[T_0,2\,T_0]$ provided that
we take the sup of the latter over all the positions of the process at time $T_0$. Mathematically, this intuition amounts to the following observations. We first remark that 
\begin{eqnarray}
\label{eq:Appendix-1}
&&P(\{ x+W_t \in E_r^c \hbox{ for all }   t \in [0,2\,T_0]\})\nonumber\\
&&=P(\{x+W_t \in E_r^c  \hbox{ for all }  t \in [0,T_0]\}\cap
\{ x+W_t \in E_r^c\hbox{ for all }  t \in [T_0,2\,T_0]\})\nonumber\\
&&=P(\{x+W_t \in E_r^c  \hbox{ for all }  t \in [0,T_0]\}\nonumber\\
&&\hskip 1truecm\cap\{ x+W_{T_0}+(W_{s+T_0}-W_{T_0}) \in E_r^c \hbox{ for all }  s \in [0,T_0]\}).
\end{eqnarray}
We focus on the latter event
\begin{equation*}
A_2:=\{ x +W_{T_0}+(W_{s+T_0}-W_{T_0}) \in E_r^c  \hbox{ for all }  s \in [0,T_0]\}.
\end{equation*}
Using the Markov property, the increment~$W_{s+T_0}-W_{T_0}$ is a Brownian motion, which is denoted by~$\widetilde{W_s}$, that is independent from the original Brownian motion~$W_t$ on~$[0,T_0]$ and thus of the event \begin{equation*}
 A_1:=\{x+W_t \in E_r^c  \hbox{ for all }  t \in [0,T_0]\}.
 \end{equation*}
  In addition, writing the event
 \begin{equation*}
  A_2=\big\{ x+W_{T_0}+\widetilde{W_s} \in E_r^c\ \hbox{ for all }  s \in [0,T_0]\},
  \end{equation*}
again by the Markov property, we note that $A_2$ depends on the event $A_1$ through the position $y=x+W_{T_0}$. 

We therefore have 
\begin{equation}
\label{eq:Appendix-2}
P (A_1\cap A_2)= P (A_1)P (A_2| A_1)=P(A_1) P (A_2| x+W_{T_0} ).
\end{equation}
Now we have
\begin{eqnarray}
\label{eq:Appendix-3}
P (A_2| x\,+\,W_{T_0})
\leq\sup_{y\in{\mathbb T}^d}P(\{y+\widetilde{W_s} \in E_r^c \hbox{ for all }  s\in [0,T_0]\}).
\end{eqnarray}
Collecting~\eqref{eq:Appendix-1}, \eqref{eq:Appendix-2} and \eqref{eq:Appendix-3}, we obtain
\begin{eqnarray}
\label{eq:Appendix-4}
&&\hskip-1cm P( \{x+W_t \in E_r^c \hbox{ for all } t \in [0,2\,T_0]\})\leq P ( x+W_t \in E_r^c  \hbox{ for all } t \in [0,T_0]\})\times\nonumber\\
&&\hskip3cm\times \sup_{y\in{\mathbb T}^d} P \big(\{y+\widetilde{W_s} \in E_r^c\, \hbox{ for all }  s \in [0,T_0]\}\big).
\end{eqnarray}
Using~\eqref{eq:Appendix-sup} for each of the two factors in the right-hand side of \eqref{eq:Appendix-4}, we have
\begin{equation*}
 \sup_{x \in \mathbb{T}^d} P(\{ x +W_t \in E_r^c\ \hbox{ for all } t \in [0,2T_0]\})\leq \Big(\frac{U_r}{T_0}\Big)^2.
\end{equation*}
An iteration of this argument gives
\begin{equation}
\label{eq:Appendix-sup-k}
 \sup_{x \in \mathbb{T}^d}P(\{ x +W_t \in E_r^c\ \hbox{ for all } t \in [0,kT_0]\})\leq \Big(\frac{U_r}{T_0}\Big)^{k}.
\end{equation}
We finally consider an arbitrary time~$T$, which we may always assume larger than $2\,U_r$ since we are interested in the behavior of~$P(\{ x +W_t \in E_r^c\ \hbox{ for all } t \in [0,T]\})$ for $T$ large. Choosing $T_0=2\,U_r$ and iterating $k$ times where $k$ is the integer part~$ k\,=\,\lfloor\frac{T}{T_0}\rfloor$,  by~\eqref{eq:Appendix-sup-k} we obtain
\begin{equation*}
 \sup_{x \in \mathbb{T}^d}P( \{x+W_t \in E_r^c\, \hbox{ for all } t \in [0,k\,T_0]\})\leq \Big(\frac12\Big)^{\left\lfloor\frac{T}{T_0}\right\rfloor}\,\leq \, \Big(\frac12\Big)^{\frac{T}{T_0}-1},
\end{equation*}
thus, since $kT_0=\lfloor\frac{T}{T_0}\rfloor\,T_0\leq T$, that, \emph{a fortiori},
\begin{equation*}
 \sup_{x \in \mathbb{T}^d}P(\{x +W_t \in E_r^c\, \hbox{ for all } t \in [0,T]\})\,\leq \, \Big(\frac12\Big)^{\frac{T}{T_0}-1}.
\end{equation*}
It follows that 
\begin{equation}
\label{eq:Appendix-sup-exp}
 \sup_{x \in \mathbb{T}^d} P( \{x+W_t \in E_r^c\ \hbox{ for all } t \in [0,T]\})\,\leq \,2\,\exp{(-\lambda \,T)}
\end{equation}
for, say, the positive constant~$\lambda\,=\,\frac{1}{2\,U_r}\log{2}$. An optimal constant (the actual exponential rate) may be determined but~\eqref{eq:Appendix-sup-exp} is sufficient for our purposes in the main body of the text: the probability~$P( \{x+W_t \in E_r^c\ \hbox{ for all } t \in [0,T]\})$ is, uniformly in~$x\in{\mathbb T}^d$, getting exponentially small as time grows. This concludes the proof of our claim.
\end{proof}

\smallskip\noindent
{\bf Acknowledgements.} The authors thank  Jean-Fran\c cois Delmas for outlining the proof of Proposition \ref{concat}, and Mauro Mariani for the proof of Proposition \ref{apbrown}. The authors are grateful to Yves Achdou for reading through a preliminary version of this manuscript and suggesting many improvements.
The first author is member of GNAMPA, INdAM.
 This article is based on work supported by the National Research Project PRIN 2022J4FYNJ  ``Variational methods for stationary and evolution problems with singularities and interfaces"
 funded by the Italian Ministry of University and Research, and on Andrea Braides's part, by the Excellence Project MatMod@TOV (Grant Agreement No.~CUP\_E83C23000330006) of the
Department of Mathematics of the University of Rome Tor Vergata. The first two authors thank INRIA for the kind hospitality in Paris.


\begin{thebibliography}{10}

\setlength{\itemsep}{0pt}  % Azzera lo spazio tra gli elementi
    \setlength{\parsep}{0pt}   % Azzera lo spazio tra i paragrafi dello stesso elemento
   

\bibitem{zbMATH07755395}
Y.~Achdou and C.~Le~Bris.
\newblock Homogenization of some periodic {Hamilton}-{Jacobi} equations with
  defects.
\newblock {\em Commun. Partial Differ. Equations}, 48(6):944--986, 2023.

\bibitem{zbMATH08099540}
Y.~Achdou and C.~Le~Bris.
\newblock Homogenization of {Hamilton}-{Jacobi} equations with defects leading
  to stratified problems.
\newblock {\em NoDEA, Nonlinear Differ. Equ. Appl.}, 32(6):39, 2025.
\newblock Id/No 119.

\bibitem{MR1484411}
M.~Bardi and I.~Capuzzo~Dolcetta.
\newblock {\em Optimal {C}ontrol and {V}iscosity {S}olutions of
  {H}amilton-{J}acobi-{B}ellman equations}.
\newblock Systems \& Control: Foundations \& Applications. Birkh\"{a}user
  Boston, Inc., Boston, MA, 1997.

\bibitem{zbMATH06190335}
G.~Barles.
\newblock An introduction to the theory of viscosity solutions for first-order
  {Hamilton}-{Jacobi} equations and applications.
\newblock In {\em Hamilton-Jacobi equations: approximations, numerical analysis
  and applications. Based on the lectures of the CIME summer school, Cetraro,
  Italy, August 29--September 3, 2011}, pages 49--109. Berlin: Springer;
  Florence: Fondazione CIME, 2013.

\bibitem{zbMATH07814183}
G.~Barles and E.~Chasseigne.
\newblock {\em On modern approaches of {Hamilton}-{Jacobi} equations and
  control problems with discontinuities. {A} guide to theory, applications, and
  some open problems}, volume 104 of {\em Prog. Nonlinear Differ. Equ. Appl.}
\newblock Cham: Birkh{\"a}user, 2024.

\bibitem{MR5057984}
A.~Braides, G.~Dal~Maso, and C.~Le~Bris.
\newblock A variational approach to the stability in the homogenization of some
  {H}amilton-{J}acobi equations.
\newblock {\em SIAM J. Math. Anal.}, 58(2):1744--1771, 2026.

\bibitem{BMS}
A.~Braides, M.~Maslennikov, and L.~Sigalotti.
\newblock Homogenization by blow-up.
\newblock {\em Applicable Anal.}, 87:1341--1356, 2008.

\bibitem{zbMATH06921166}
P.~Cardaliaguet, C.~Le~Bris, and P.~E. Souganidis.
\newblock Perturbation problems in homogenization of {Hamilton}-{Jacobi}
  equations.
\newblock {\em J. Math. Pures Appl. (9)}, 117:221--262, 2018.

\bibitem{MR1159184}
L.~C. Evans.
\newblock Periodic homogenisation of certain fully nonlinear partial
  differential equations.
\newblock {\em Proc. Roy. Soc. Edinburgh Sect. A}, 120(3-4):245--265, 1992.

\bibitem{MR2179357}
W.~H. Fleming and H.~M. Soner.
\newblock {\em Controlled {M}arkov processes and viscosity solutions},
  volume~25 of {\em Stochastic Modelling and Applied Probability}.
\newblock Springer, New York, 2006.

\bibitem{FM}
I.~Fonseca and S.~M\"uller.
\newblock Quasiconvex integrands and lower semicontinuity in {$L^1$}.
\newblock {\em SIAM J. Math. Anal.}, 23:1081--1098, 1992.

\bibitem{PLL-college}
P.-L. Lions.
\newblock {\em {\'E}quations elliptiques ou paraboliques, et
  homog\'en\'eisation pr\'e\-ci\-s\'ee}, {L}ectures at {C}oll\`ege de {F}rance
  2013/14. {V}ideo recordings available at: {\tt
  www.college-de-france.fr/site/pierre-louis-lions}.

\bibitem{Lions-Sou}
P.-L. Lions and P.~E. Souganidis.
\newblock Homogenization of ``viscous'' {H}amilton-{J}acobi equations in
  stationary ergodic media.
\newblock {\em Comm. Partial Differential Equations}, 30(1-3):335--375, 2005.

\bibitem{zbMATH06042622}
N.~Touzi.
\newblock {\em Optimal stochastic control, stochastic target problems, and
  backward {SDE}.}, volume~29 of {\em Fields Inst. Monogr.}
\newblock New York, NY: Springer; Toronto: The Fields Institute for Research in
  the Mathematical Sciences, 2013.

\bibitem{zbMATH01325009}
J.~Yong and X.~Y. Zhou.
\newblock {\em Stochastic controls. {Hamiltonian} systems and {HJB} equations},
  volume~43 of {\em Appl. Math. (N.Y.)}.
\newblock New York, NY: Springer, 1999.

\end{thebibliography}
\end{document}